\documentclass[11pt,twoside,reqno]{amsart}
\usepackage[colorlinks=true,citecolor=black]{hyperref}
\usepackage{mathptmx, amsmath, amssymb, amsfonts, amsthm, mathptmx, enumerate, color}
\usepackage{graphicx}
\usepackage{epstopdf}
\newtheorem{theorem}{Theorem}[section]
\newtheorem{lemma}[theorem]{Lemma}
\newtheorem{proposition}[theorem]{Proposition}

\theoremstyle{definition}

\numberwithin{equation}{section}

\usepackage{hyperref}
\hypersetup{
	colorlinks=true,
	linkcolor=black,
	filecolor=magenta,      
	urlcolor=cyan,
	pdftitle={Overleaf Example},
	pdfpagemode=FullScreen,
}

\begin{document}
\setcounter{page}{1}

\vspace*{2.0cm}
\title[Coupled system of generalized KDV equations]
{Spatial analyticity of solutions for a coupled system of generalized KdV equations}
\author[A. Atmani, A. Boukarou,  D. Benterki,  and Kh. Zennir]{A. Atmani, A. Boukarou,  D. Benterki,  and Kh. Zennir}
\maketitle
\vspace*{-0.6cm}

\vskip 4mm {\footnotesize \noindent {\bf Abstract.}
The solution of a coupled system consisting of generalized Korteweg-de Vries-type equations is obtained for all time where the initial data are analytic on a band in the complex plane.  We show that the width of this band decreases algebraically with time. 	
	

 }

\renewcommand{\thefootnote}{}

\tableofcontents

 \section{Introduction}
This paper deals with the initial-value problem  for a coupled system of generalized Korteweg–de Vries (gKdV) equation
\begin{equation}\label{p01}
	\left\{\begin{array}{l}
		u_{t}+\partial_{x}^{3} u+\partial_{x}\left(u^{p} v^{p+1}\right)=0  \\
		v_{t}+\partial_{x}^{3} v+\partial_{x}\left(u^{p+1} v^{p}\right)=0, \quad x, t \in \mathbb{R}, p \in \mathbb{Z}^{+} \\
		u(x, 0)=u_{0}(x), \quad v(x, 0)=v_{0}(x),
	\end{array}\right.
\end{equation}
where the unknown  $ u = u(x, t),~ v = v(x, t)$  and the initial data$ ( u_{0}(x), v_{0}(x) ) $ are real-valued.\\
This type of equation is a special case of an important vast class of nonlinear evolution equations which was studied by  M. Ablowitz \cite{17}, and it has applications in physical  problems, which  describes the strong interaction of two dimensional long internal gravity waves.\\
For $ p=1  $, the system can be reduce to a  coupled system of modified KdV (mKdV) equations
\begin{equation}\label{p08}
	\left\{\begin{array}{l}
		u_{t}+\partial_{x}^{3} u+\partial_{x}\left(u v^{2}\right)=0  \\
		v_{t}+\partial_{x}^{3} v+\partial_{x}\left(u^{2} v\right)=0, \quad x, t \in \mathbb{R} \\
		u(x, 0)=u_{0}(x), \quad v(x, 0)=v_{0}(x).
	\end{array}\right.
\end{equation}
 Here, the author proved the local well posdness in in $H^s, s \geq \frac{1}{4}$. For $ s \geq 1  $, it is proved that the global well posdness is assured. In addition,  M. Panthee improved it to extend solution  to be in any time interval  $[0, T] $  for $ s > \frac{4}{9}$.\\
 The authors  in \cite{20} studied the local well-posdness   in $  (H^{s} \times  H^{s} ) $  with $ s > -\frac{1}{2} $ for   system consisting modified Korteweg–de Vries-type equations
 \begin{equation}
 \left\{\begin{array}{l}
 u_{t}+\partial_{x}^{3} u+\partial_{x}\left(u^{2} v^{3}\right)=0, \\
 v_{t}+ \alpha \partial_{x}^{3} v+\partial_{x}\left(u^{3} v^{2}\right)=0, \quad x, t \in \mathbb{R} \\
 u(x, 0)=u_{0}(x), \quad v(x, 0)=v_{0}(x),
 \end{array}\right.
 \end{equation}
 where $ 0 < \alpha <  1    $  and $  (u_{0}, v_{0})   $ is given  in low regularity Sobolev spaces$  (H^{s} \times  H^{s} )$, but if  $  \alpha = 1   $ the authors obtained the local well posedness for $ s \geqslant \frac{1}{4}$.\\
In \cite{7}, the  problem $ (\ref{p01})  $ is  studied and the local and global well-posedness  results with $ (u_{0}, v_{0}) \in H^{s} \times H^{s}   $, $ s\geqslant 1  $  and $ p \geqslant 1 $ is shown. The golobal well-posedness was obtained by using the next conserved quantities satisfied by the flow of $ (\ref{p01}) $
$$   \int_{\mathbb{R}  }u dx \quad\int_{\mathbb{R} }    v dx  \quad \frac{1}{2} \int_{\mathbb{R}}u^{2}+v^{2} dx
\quad  \textit{and}
   \quad  \frac{1}{2} \int_{\mathbb{R}  } u_{x}^{2}+v_{x}^{2} -\frac{2}{p+1} u^{p+1} v^{p+1} dx.$$ 
In addition,  the authors showed  the existence and nonlinear stability of the solitary wave solution. The study of stability for  solitary wave solution is followed from the abstract results of Grillakis, for more details, please see  \cite{9,4,14,18}.\\
 For $  p =2   $, the system is turn out to a coupled system of modified Korteweg–de Vries (gKdV) equation 
\begin{equation}
\left\{\begin{array}{l}\label{p07}
u_{t}+\partial_{x}^{3} u+\partial_{x}\left(u^{2} v^{3}\right)=0, \\
v_{t}+\partial_{x}^{3} v+\partial_{x}\left(u^{3} v^{2}\right)=0, \quad x, t \in \mathbb{R} \\
u(x, 0)=u_{0}(x), \quad v(x, 0)=v_{0}(x).
\end{array}\right.
\end{equation}
Panthee and Scialom  \cite{5}, investigated some well-posedness issues for eq $ (\ref{p07})   $
in $H^{s} \times H^{s}$, which proved local and global will posdness for $ s\geqslant 0$.\\
 For related problems in analytic Gevrey spaces, we review the results  in $ 2 D $ by M. Shan, L. Zhang  \cite{21}, where the authors proved that the following problem (the Cauchy problem associated with the  $2 D $  generalized Zakharov-Kuznetsov equation)
\begin{equation}
\left\{\begin{array}{l}\label{p098}
u_{t}+ (\partial_{x}^{3} +\partial_{y}^{3} )   u + (\partial_{x} + \partial_{y} ) u^{p+1} =0, \\
u(0, x,y)=u_{0}(x, y),
\end{array}\right.
\end{equation}
 has an analytic  solutions  in a strip the width, and  they gave an  algebraic lower bounds. \\
Bona and Gruji\'c \cite{22}   showed the well-posedness of a KdV-type Boussinesq
system
 \begin{equation}
\left\{\begin{array}{l}\label{p09}
u_{t}+ v_{x}  + u u_{x}  +   v_{xxx}    = 0 \\
v_{t}+ u_{x}  + ( u v)_{x}  +   u_{xxx} =  0.
\end{array}\right.
\end{equation}
There is another method in this direction, we mention the works by A. Boukarou et {\em al.} in the next series of papers \cite{BoukarouA2,Boukarou2,Boukarou3,Boukarou4,Boukarou6,Boukarouarxiv,Boukarou5}.\\
Motivated by the previouse results, we consider our main ptoblem with initial data are analytic on a band in the complex plane and obtained solution for all time.  We also showed that the width of this band decreases algebraically with time. 	\\
This paper is continuation of our prevouse results and it is structured as follows. In section $ 1$,  we give some historical review and motivate this paper to further strengthened, and  innovate the main contributions and introduce our main results which we will prove later (local and global well posedness of equation $ \eqref{p01}$). In section $ 2$, we present some definition and the necessary function spaces such as the analytic function spaces  $ \mathcal{G}_{\rho, s}  $,  analytic Bourgain space  $X_{\rho, s, b}  $ which will be used. In section $3$, we prove the Linear and Bilinear Estimates  which needed to prove the main results. In section $4$, we prove the local and global  well-posdness  and then obtained lower bound. \\
We provide a clear, sober and well-written analysis of the problem.
\begin{theorem}\label{the1.2}
	Let $ s > \frac{3}{2} $ and $  p \geq 1  $ and for initial data $(u_0,v_0)\in
	\mathcal{G}_{\rho, s}\times \mathcal{G}_{\rho, s}$, $ \rho > 0  $, there exists a positive time $  T $, such that the initial -value problem (\ref{p01}) is well-posed in the space
	$$C \left( \left[ 0, T  \right];\mathcal{G}_{\rho, s} \right)\times C \left( \left[ 0, T  \right];\mathcal{G}_{\rho, s} \right).  $$
	
\end{theorem}

\begin{theorem}\label{th03}
	Let $  \rho_{0} > 0 $ and $  s > \frac{3}{2} $ and let $ T \geq t_{0} $ suppose that  the solution $  u $, $ v $
	given by Theorem $ (\ref{the1.2})  $ extends globally in time.  Then, we have
	$$ (u,v)\in C ([0, 2T  ], \mathcal{G}_{\rho(T ) /2, s}  )\times C ([0, 2T  ], \mathcal{G}_{\rho(T ) /2, s}  ),$$
	where $ \rho(T ) $  is given by
	$$ \rho (t) =  \min \left\lbrace   \rho_{1},  K T ^{-2p^{2} -6p-1}  \right\rbrace.   $$
	for some constant $  K > 0 $.
	
\end{theorem}

 \section{Preliminary estimates and Function spaces}
The $\widehat{ u}$ is denote the Fourier transform of  $ u $ which is defined as 
$$
\widehat{ u}(\zeta) =   \frac{1}{\sqrt{2\pi }}\int _{-\infty}^{+ \infty} u(x) e^{-ix \zeta}  dx. $$
For a function $  u(x, t) $ of two variable  we have 
$$
\widehat{u}^{x} (\zeta, t) =   \frac{1}{\sqrt{2\pi }}\int _{-\infty}^{+ \infty} u(x, t) e^{-ix \zeta}  dx,
$$
and
$$
\widehat{u} (\zeta,\eta ) =   \frac{1}{2\pi }\int _{-\infty}^{+ \infty} \int _{-\infty}^{+ \infty} u(x, t) e^{-ix \zeta} e^{-it \eta} dx d\eta.
$$ 
We  note that the operators  $ A, \Lambda  $ and $  F_{\rho }$ are defined as
\[
\widehat{Au} (\zeta,\eta ) = \left(  1+ | \zeta | \right)\widehat{u}(\zeta,\eta );
\]
\[
\widehat{\Lambda u} (\zeta,\eta ) = \left(  1+ | \eta | \right)\widehat{u}(\zeta,\eta );
\]
\[
 \widehat{F_{\kappa} }(\zeta, \eta ) =  \frac{f(\zeta, \eta  )}{\left(1+| \eta - \zeta^{3} |\right)^{\kappa} }.
\]
The  mixed  $ L^{p}- L^{q} $ -norm  is defined by 
$$
\| u \|_{L^{p} L^{q} }  =  \left( \int _{-\infty}^{+ \infty} \left| \int _{-\infty}^{+ \infty}  |u(x, t)|^{q }  dt\right|^{\frac{p}{q} }    dx  \right)^{\frac{1}{p} }
$$ 	
The analytic  Gevrey class $ \mathcal{G}_{\rho,s}$ is defined by Foias and Temam \cite{T1} as
\[
\left\|u_{0}\right\|_{\mathcal{G}_{\rho, s}}^{2}=\|\mathrm{e}^{ \rho(1+|\zeta|)}(1+|\zeta|)^{s} \widehat{u_{0}}(\zeta) \|_{L^{2}_{\zeta}}.
\]
For     $ s,  \in \mathbb{R} $, $b \in[-1,1]$  and  $ \rho > 0  $, we denote $  X_{\rho,s,b}$ by $\|\cdot\|_{\rho,s, b}$  with respect to the norm

		\[
		\| u \|_{X_{\rho,s,b}} =
		\bigg\|   e^{\rho  ( 1+ |\zeta  |)} (1+|  \zeta  | )^{s}  (1+\left|\eta-\zeta^{3}\right|)^{b}   \hat{u}( \zeta,\eta )        \bigg\|_{L^{2}_{\zeta,  \eta}}.
		\]
For $\rho=0, X_{\rho, s, b}$ coincides with the space $X_{s, b}$ introduced by Bourgain \cite{L1}, and Kenig, Ponce and Vega \cite{A1}. The norm of $X_{s, b}$ is denoted by $\|\cdot\|_{s, b}$, as follow
\[
\| u \|_{X_{s,b}} =
\bigg\|  (1+|  \zeta  | )^{s}  (1+\left|\eta-\zeta^{3}\right|)^{b}   \hat{u}( \zeta,\eta )        \bigg\|_{L^{2}_{\zeta,  \eta}}.
\]
\section{Linear and Multilinear Estimates}
In this section, we shall deduce several estimates to be used in the proof of Theorem (\ref{the1.2}). 
\begin{lemma}\label{lem3}
	Let  $  0 < \sigma  < \rho $ and $ n \in \mathbb{N}  $. Then, we have 
	\begin{align*}
	\sup_{\substack{ x+iy \in S_{\rho- \sigma } }} \vert \partial^{n}_{x} u(x+iy)  \vert \leq C  \Vert  u \Vert_{\mathcal{G}_{\rho}},
	\end{align*}
	where $ C $ is constant depending on   $ \zeta $  and $  n $.
\end{lemma}

\begin{lemma}\label{lem2}
	Let $  b > \frac{1}{2} $,~~$  s \in \mathbb{R}$
	~~and $ \rho \geq 0 $, then for all $   T > 0 $,~
	we have 
	\begin{equation*}
	X_{\rho, s, b } \hookrightarrow C\left([0,T], G^{\rho, s} \right).
	\end{equation*}
\end{lemma}
\begin{proof}
	We  define the operator $ \Theta $ \\
	$$  \widehat{\Theta u}^{x}(\zeta, t) = e^{\rho (1+| \zeta |)}\widehat{u}^{x}(\zeta, t),    $$
	satisfy
	\begin{equation*}
	\|   u   \|_{X_{\rho, s, b}} = \|   \Theta  u   \|_{X_{ s, b}},
	\end{equation*}
	and
	\begin{equation*}
	\|   u   \|_{\mathcal{G}_{\rho, s }} = \|   \Theta u   \|_{H^{ s}}.
	\end{equation*}
	 We observe that $    \Theta u $ belongs to $ C([0,T  ], H^{s})$  and for
	some $ C > 0 $ we have
	\begin{equation*}
	\|   \Theta u   \|_{C\left([0,T], H^{s} \right)} \leq C ~\|   \Theta u  \|_{X_{ s, b}}.
	\end{equation*}
	Thus, it follows that $u \in C\left([0,T], G^{\rho, s} \right)  $ and
	\begin{equation*}
	\|   u   \|_{C\left([0,T],  G^{\rho, s}\right)} \leq C ~\|   u  \|_{X_{   \rho, s, b}}.
	\end{equation*}
\end{proof}
 By using Duhamel's formula (\ref{p01}), we may write the solution
\begin{equation*}\left\{
\begin{array}{l}
u(x,t) =   W(t)u_{0}(x  )   - \displaystyle\int_{0}^{t}  W(t-t^{\prime} )w_{1}(x, t^{\prime})dt^{\prime}, \\ \\
v(x,t)=   W(t)v_{0} (x  )   - \displaystyle\int_{0}^{t}  W(t-t^{\prime})w_{2}(x, t^{\prime})dt^{\prime},
\end{array}\right.
\end{equation*} \\
where $W(t)= e^{-t\partial_{x}^{3}}$, $w_{1} = \partial_{x}\left(u^{p} v^{p+1}\right) $ and
$w_{2} = \partial_{x}\left(u^{p+1} v^{p}\right)$.\\
Next, we localize in time variable by using a cut-off function
$ \psi(t)  \in  C_{0}^{\infty} (-2,2)$  with\\
$   0\leq \psi(t) \leq 1,  \psi(t) =1      $
on $[-1,1]$ and for
$0 < T < 1$.\\
We define $\psi_{T}(t)= \psi(\frac{t}{T}) $, where
\begin{equation*}
\left\{\begin{array}{l}
\psi \in C_{0}^{\infty}, \psi = 1 \quad  in \big[  -1; 1 \big]
\\
supp \psi\subset  \big[  -2; 2 \big]\\
\psi_{T}(t) = \psi((\frac{t}{T})).
\end{array}\right.
\end{equation*}
We consider the operator $  \varXi  $, $  \Gamma $ given by the following
\begin{equation}\label{p5}
\left\{\begin{array}{l}
\varXi (t)	= \psi (t) W(t)u_{0}- \psi_{T}(t) \displaystyle \int_{0}^{t} W(t-t^{\prime})w_{1}(t^{\prime}) dt^{\prime} \\ \\
\Gamma(t)= \psi (t)W(t)v_{0}- \psi_{T}(t)\displaystyle\int_{0}^{t} W(t-t^{\prime})w_{2}(t^{\prime}) dt^{\prime}.
\end{array}\right.
\end{equation}
 We start with the following useful Lemma.
\begin{lemma}\label{lem1}\cite{A1,13}
	Let $ \rho \geq 0 $, $b > \frac{1}{2}$,~~$ b-1 <  b'  < 0 $, and $ T \geq 1 $. Then there exist a constant $ c $ such that the following estimates holds 
	\begin{equation}\label{eq1}
	\|  \psi(t) W(t)u_{0}\|_{\rho, s, b}  \leq c T^{\frac{1}{2}} \| u_{0} \|_{\mathcal{G}_{\rho, s }},
	\quad \quad
	\|  \psi(t) W(t)v_{0}\|_{\rho, s, b}  \leq c T^{\frac{1}{2}} \| v_{0} \|_{\mathcal{G}_{\rho, s }},
	\end{equation}
	and
	\begin{equation}\label{eq3}
	\|  \psi_{T}(t) u\|_{\rho, s, b }  \leq c   \| u\|_{\rho, s, b},
	\quad \quad
		\|  \psi_{T}(t) v\|_{\rho, s, b}  \leq c   \| v\|_{\rho, s, b},
	\end{equation}
	and
	\begin{equation}\label{eq4}
	\|   \psi_{T}(t) \int_{0}^{t}  W(t-s)w(s) ds\|_{\rho, s, b}  \leq c  T\| w\|_{\rho, s, b'}.
	\end{equation}
\end{lemma} 
\begin{lemma}\label{2.3}
	(\cite{13,4})
	Let $ s $ and $ \kappa $ be given. There is a constant $ c $ depending on $ s $ and $\kappa  $ such that 		
	\begin{equation}\label{102}
	If \quad \kappa > \frac{1}{4}, \quad then \quad
	\| A^{\frac{1}{2}} F_{\kappa} \|_{L_{x}^{4} L_{t}^{2} }  \leq C \| f \|_{L^{2}_{\zeta} L^{2}_{\eta} },
	\end{equation} 
	
	\begin{equation}\label{103}
	If \quad \kappa > \frac{1}{4},\quad then \quad
	\| A F_{\kappa} \|_{L_{x}^{\infty} L_{t}^{2} }  \leq C \| f \|_{L^{2}_{\zeta} L^{2}_{\eta} },
	\end{equation} 
	
	\begin{equation}\label{104}
	\textit{If} \quad  \kappa > \frac{1}{2}, \quad \textit{and}\quad s >3\kappa ,\quad\textit{then}  \quad
	\| A^{-s} F_{\kappa} \|_{L_{x}^{2} L_{t}^{\infty}  }  \leq C \| f \|_{L^{2}_{\zeta} L^{2}_{\eta} },
	\end{equation} 
	
	\begin{equation}\label{105}
	\textit{If} \quad  \kappa > \frac{1}{2}, \quad \textit{and}\quad s >\frac{1}{4},\quad\textit{then}  \quad
	\| A^{-s} F_{\kappa} \|_{L_{x}^{4} L_{t}^{\infty}  }  \leq C \| f \|_{L^{2}_{\zeta} L^{2}_{\eta} },
	\end{equation} 
	
	\begin{equation}\label{106}
	If \quad \kappa > \frac{1}{2}, \quad and \quad s > \frac{1}{2}, \quad\textit{then}  \quad
	\| A^{-s} F_{\kappa} \|_{L_{x}^{\infty} L_{t}^{\infty}  }  \leq  C \| f \|_{L^{2}_{\zeta} L^{2}_{\eta} }.
	\end{equation}
\end{lemma}

\begin{lemma}\label{1.3}
	Let $ b >\frac{1}{2}  $, $ b ' < - \frac{1}{4} $, and $  s \geq 3b $. Let $ p \in \mathbb{N}$ and suppose $ u_{1},...,u_{p+1}, v_{1},..., v_{p+1}  \in   X_{\rho, s, b}$. Then there exists a constants $ c $   such that
	\begin{equation}\label{eq25}
	\| \partial_{x}\prod _{i=1}^{p} u_{i} \prod _{j=1}^{p+1} v_{j}  \|_{\rho, s, b'}\leq
	C  \prod _{i=1}^{p}\|  u_{i}  \|_{\rho, s, b}. \prod _{j=1}^{p+1}\|  v_{j}  \|_{\rho, s, b},
	\end{equation}
	\begin{equation}\label{eq28}
	\| \partial_{x}\prod _{i=1}^{p+1} u_{i} \prod _{j=1}^{p} v_{j}  \|_{\rho, s, b'}\leq
	C  \prod _{i=1}^{p+1}\|  u_{i}  \|_{\rho, s, b}. \prod _{j=1}^{p}\|  v_{j}  \|_{\rho, s, b}.
	\end{equation}
\end{lemma}
\begin{proof}
	First of all, for $  i= 1,2,...,p+1$   and $j=1,2,...,p+1 $, we define
	\begin{align*}
	f_{i} (\zeta, \eta ) = (1+  | \zeta |)^{s} (1+  |\eta - \zeta^{3} |)^{b} e^{\rho (1+  | \zeta |)} | \widehat{u_{i} }(\zeta,\eta )|  \\
	g_{j} (\zeta, \eta ) = (1+  | \zeta |)^{s} (1+  |\eta - \zeta^{3} |)^{b} e^{\rho (1+  | \zeta |)  }| \widehat{v_{j} }(\zeta,\eta )|.
	\end{align*}
	The proof is first given for the case  $ p= 1$, after which the proof for a general $ 2p+1 $ will be more transparent,	that means we prove 
	\begin{equation*}
	      \| \partial_{x} u_{1}  v_{1} v_{2}  \|_{\rho, s, b'}\leq
	C \| u_{1}  \|_{\rho, s, b} \|v_{1} \|_{\rho, s, b}  \|v_{2} \|_{\rho, s, b}
	\end{equation*}
	\begin{equation*}
	      \| \partial_{x} u_{1}  u_{2} v_{1}  \|_{\rho, s, b'}\leq
	C \| u_{1}  \|_{\rho, s, b} \|u_{2} \|_{\rho, s, b}  \|v_{1} \|_{\rho, s, b}.
	\end{equation*}
	We have	
	\begin{align*}
 \| \partial_{x} u_{1}  v_{1} v_{2}  \|_{\rho, s, b'}&=  \left\|  (1+  | \zeta |)^{s} (1+  |\eta - \zeta^{3} |)^{b'} e^{\rho (1+| \zeta |)} | \widehat{\partial_{x} u_{1}  v_{1} v_{2} }(\zeta,\eta )|     \right\|_{L^{2}_{\zeta} L^{2}_{\eta} }
\\	\\& = \left\|  (1+  | \zeta |)^{s} e^{\rho (1+| \zeta |)}(1+  |\eta - \zeta^{3} |)^{b'}  |   \zeta| ~~ | \widehat{u_{1}  v_{1} v_{2} }(\zeta,\eta )|     \right\|_{{L^{2}_{\zeta} L^{2}_{\eta} }}\\ \\&	
	= \left\| (1+  | \zeta |)^{s}e^{\rho (1+| \zeta |)} (1+  |\eta - \zeta^{3} |)^{b'}    | \zeta|~~ | \widehat{u_{1}}\ast \widehat{v_{1}} \ast \widehat{v_{2}}(\zeta,\eta )|     \right\|_{L^{2}_{\zeta} L^{2}_{\eta} } \\ \\&
	= \|  (1+  | \zeta |)^{s}e^{\rho (1+| \zeta |)} (1+  |\eta - \zeta^{3} |)^{b'}    | \zeta |  \int_{\mathbb{R}^{4}}  \widehat{u_{1}}(\zeta_{1},\eta_{1} ) \widehat{v_{1}}(\zeta-\zeta_{2},\eta - \eta_{2} )\\ \\& \quad \times
	\widehat{v_{2}}(\zeta_{2}-\zeta_{1},\eta_{2}-\eta_{1} )| d\zeta_{1} d\eta_{1} d\zeta_{2} d\eta_{2}   \|_{{L^{2}_{\zeta} L^{2}_{\eta} }}\\ \\&	
=  \|  (1+  | \zeta |)^{s} e^{\rho (1+| \zeta |)}(1+  |\eta - \zeta^{3} |)^{b'}   | \zeta |  \int_{\mathbb{R}^{4}} ~~
	\left (\frac{(1+  | \zeta_{1} |)^{-s}  e^{-\rho (1+  | \zeta_{1} |)} \widehat{f_{1}}(\zeta_{1},\eta_{1} )}{(1+  | \eta - \zeta^{3} |)^{b}}\right) \\ \\& \quad \times
	\left (\frac{(1+  | \zeta-\zeta_{2}  |)^{-s}  e^{-\rho (1+  |\zeta-\zeta_{2} |)} \widehat{g_{1}}(\zeta-\zeta_{2},\eta - \eta_{2})}{(1+  | ( \eta - \eta_{2}  ) - (\zeta-\zeta_{2})^{3} |)^{b}} \right)\\ \\&\quad \times
	\left (\frac{(1+  |  \zeta_{2} -\zeta_{1} |)^{-s}  e^{-\rho (1+  | \zeta_{2} -\zeta_{1}|)} \widehat{g_{2}}(\zeta_{2} -\zeta_{1},\eta_{2}-\eta_{1} )}{(1+  |\eta_{2}-\eta_{1}  -(\zeta_{2} -\zeta_{1} )^{3} |)^{b}}\right)   d \mu \|_{{L^{2}_{\zeta} L^{2}_{\eta} }},	
	\end{align*}
	where $    d \mu =   d\zeta_{1} d\eta_{1} d\zeta_{2} d\eta_{2}  d\zeta d\eta$.	\\
	 By using the duality, we proof this estimate, where $  m(\zeta,\eta) $  is a positive function in $  L^{2}( \mathbb{R}^{2})$ with norm $ \| m\|_{ L^{2}( \mathbb{R}^{2})}=1   $, then 
	\begin{align*}
	 \| \partial_{x} u_{1}  v_{1} v_{2}  \|_{\rho, s, b'}	& \leqslant \int_{\mathbb{R}^{6}}\frac{e^{\rho (1+| \zeta |)}(1+  | \zeta |)^{s} | \zeta | m(\zeta,\eta)} {(1+  |\eta - \zeta^{3} |)^{-b'}}~~~
	\frac{e^{-\rho (1+  | \zeta_{1} |)}(1+  | \zeta_{1} |)^{-s} f_{1} (\zeta_{1},\eta_{1})} {(1+  |\eta_{1} - \zeta_{1}^{3} |)^{b}} \\ \\&
	\frac{e^{-\rho (1+  | \zeta-\zeta_{2}   |)}(1+  | \zeta-\zeta_{2}  |)^{-s} g_{1}(\zeta-\zeta_{2},\eta-\eta_{2})} {(1+  |\eta -\eta_{2} - (\zeta-\zeta_{2})^{3} |)^{b}}~~
\\	\\&
	\frac{e^{-\rho (1+  | \zeta_{2}-\zeta_{1}   |)}(1+  | \zeta_{2}-\zeta_{1}  |)^{-s} g_{2}(\zeta_{2}-\zeta_{1},\eta_{2}-\eta_{1})} {(1+  |\eta_{2}-\eta_{1} - (\zeta_{2}-\zeta_{1})^{3} |)^{b}}  d \mu.
	\end{align*}
	Using the inequality 
	\begin{equation*}
	| \zeta| \leq   | \zeta_{1}| +  | \zeta- \zeta_{2}| +   |  \zeta_{2}- \zeta_{1}| \quad
\text{then}	\quad
	e^{\rho (1+| \zeta |)}  \leq  e^{\rho (1+  | \zeta_{1} |)}  \times e^{\rho (1+  | \zeta-\zeta_{2}   |)}   \times e^{\rho (1+  | \zeta_{2}-\zeta_{1}   |)}.
	\end{equation*}
	Then 
	\begin{align*}
	 \| \partial_{x} u_{1}  v_{1} v_{2}  \|_{\rho, s, b'}	& \leqslant \int_{\mathbb{R}^{6}}\frac{(1+  | \zeta |)^{s} | \zeta | m(\zeta,\eta)} {(1+  |\eta - \zeta^{3} |)^{-b'}}
	\frac{(1+  | \zeta_{1} |)^{-s} f_{1} (\zeta_{1},\eta_{1})} {(1+  |\eta_{1} - \zeta_{1}^{3} |)^{b}}
	\frac{(1+  | \zeta-\zeta_{2}  |)^{-s} g_{1}(\zeta-\zeta_{2},\eta-\eta_{2})} {(1+  |\eta -\eta_{2} - (\zeta-\zeta_{2})^{3} |)^{b}} \\ \\&
	\frac{(1+  | \zeta_{2}-\zeta_{1}  |)^{-s} g_{2}(\zeta_{2}-\zeta_{1},\eta_{2}-\eta_{1})} {(1+  |\eta_{2}-\eta_{1} - (\zeta_{2}-\zeta_{1})^{3} |)^{b}} d\mu.
	\end{align*}
	Now, split the Fourier space into six regions as follow 
		\begin{enumerate}\label{ta14}
			
		\item $|\zeta-\zeta_{2}  | \leq | \zeta_{2}-\zeta_{1}  |\leq | \zeta_{1} | $
		\item $| \zeta-\zeta_{2}  | \leq| \zeta_{1} | \leq | \zeta_{2}-\zeta_{1}  | $
		\item $ | \zeta_{1} |  \leq | \zeta_{2}-\zeta_{1}  |\leq| \zeta-\zeta_{2}  | $
		\item $ | \zeta_{1} |  \leq  | \zeta-\zeta_{2}  | \leq | \zeta_{2}-\zeta_{1}  | $
		\item $ |  \zeta_{2}-\zeta_{1}  |\leq| \zeta-\zeta_{2}   \leq | \zeta_{1} | $
		\item $ |  \zeta_{2}-\zeta_{1}  | \leq | \zeta_{1} | \leq | \zeta-\zeta_{2}  |. $
	
	\end{enumerate}
	We begin by the case $ (1) $
	$$
	|\zeta-\zeta_{2}  | \leq | \zeta_{2}-\zeta_{1}  |\leq | \zeta_{1} |.
	$$
	Then
	\begin{equation}\label{eq55}
	(1+ |\zeta-\zeta_{2}  |)^{-s} \geq (1+| \zeta_{2}-\zeta_{1}  | )^{-s} \geq (1+ | \zeta_{1} | )^{-s},
	\end{equation}
and,   we assume that $  | \zeta| \leq 1  $
	or    $  | \zeta| \geq 1  $. \\
	 Firstly, by $  | \zeta| \geq 1   $, then
	$$ (1+ | \zeta| )^{s} \leq  (| \zeta|+ | \zeta| )^{s}= 2^{s}(| \zeta|)^{s} = C  (| \zeta|)^{s}.   $$
	 By the last inequality and $  (\ref{eq55})$, we obtain 
	\begin{align*}
		 \| \partial_{x} u_{1}  v_{1} v_{2}  \|_{\rho, s, b'} & \leqslant\int_{\mathbb{R}^{6}}\frac{(1+  | \zeta |)^{s} | \zeta | m(\zeta,\eta)} {(1+  |\eta - \zeta^{3} |)^{-b'}}~
	\frac{(1+  | \zeta_{1} |)^{-s} f_{1} (\zeta_{1},\eta_{1})} {(1+  |\eta_{1} - \zeta_{1}^{3} |)^{b}}~
	\frac{(1+  | \zeta-\zeta_{2}  |)^{-s} g_{1}(\zeta-\zeta_{2},\eta-\eta_{2})} {(1+  |\eta -\eta_{2} - (\zeta-\zeta_{2})^{3} |)^{b}}\\& \quad \times
	\frac{(1+  | \zeta_{2}-\zeta_{1}  |)^{-s} g_{2}(\zeta_{2}-\zeta_{1},\eta_{2}-\eta_{1})} {(1+  |\eta_{2}-\eta_{1} - (\zeta_{2}-\zeta_{1})^{3} |)^{b}}d\mu
	\\&
	\leq C
	\int_{\mathbb{R}^{6}}\frac{( | \zeta |)^{s} | \zeta | m(\zeta,\eta)} {(1+  |\eta - \zeta^{3} |)^{-b'}}~
	\frac{(1+  | \zeta_{1} |)^{-s} f_{1} (\zeta_{1},\eta_{1})} {(1+  |\eta_{1} - \zeta_{1}^{3} |)^{b}}~
	\frac{(1+  | \zeta-\zeta_{2}  |)^{-s} g_{1}(\zeta-\zeta_{2},\eta-\eta_{2})} {(1+  |\eta -\eta_{2} - (\zeta-\zeta_{2})^{3} |)^{b}}~
	\\& \quad \times
	\frac{(1+  | \zeta_{2}-\zeta_{1}  |)^{-s} g_{2}(\zeta_{2}-\zeta_{1},\eta_{2}-\eta_{1})} {(1+  |\eta_{2}-\eta_{1} - (\zeta_{2}-\zeta_{1})^{3} |)^{b}} d\mu \\&
	\leq C
	\int_{\mathbb{R}^{6}}\frac{( | \zeta |)^{s+1}  m(\zeta,\eta)} {(1+  |\eta - \zeta^{3} |)^{-b'}}~~
	\frac{(1+  | \zeta_{1} |)^{-s} f_{1} (\zeta_{1},\eta_{1})} {(1+  |\eta_{1} - \zeta_{1}^{3} |)^{b}}~
	\frac{(1+  | \zeta-\zeta_{2}  |)^{-s} g_{1}(\zeta-\zeta_{2},\eta-\eta_{2})} {(1+  |\eta -\eta_{2} - (\zeta-\zeta_{2})^{3} |)^{b}}~
	\\& \quad \times
	\frac{(1+  | \zeta_{2}-\zeta_{1}  |)^{-s} g_{2}(\zeta_{2}-\zeta_{1},\eta_{2}-\eta_{1})} {(1+  |\eta_{2}-\eta_{1} - (\zeta_{2}-\zeta_{1})^{3} |)^{b}} d\mu,
	\end{align*}
	then
	\begin{align*}
	 	 \| \partial_{x} u_{1}  v_{1} v_{2}  \|_{\rho, s, b'} & \leq C
	\int_{\mathbb{R}^{6}}\frac{( | \zeta |)^{\frac{1}{2}}  m(\zeta,\eta)} {(1+  |\eta - \zeta^{3} |)^{-b'}}
	\frac{(1+  | \zeta_{1} |)^{\frac{1}{2}} f_{1} (\zeta_{1},\eta_{1})} {(1+  |\eta_{1} - \zeta_{1}^{3} |)^{b}}
	\frac{(1+  | \zeta-\zeta_{2}  |)^{-s} g_{1}(\zeta-\zeta_{2},\eta-\eta_{2})} {(1+  |\eta -\eta_{2} - (\zeta-\zeta_{2})^{3} |)^{b}}
	\\&
	\frac{(1+  | \zeta_{2}-\zeta_{1}  |)^{-s} g_{2}(\zeta_{2}-\zeta_{1},\eta_{2}-\eta_{1})} {(1+  |\eta_{2}-\eta_{1} - (\zeta_{2}-\zeta_{1})^{3} |)^{b}} d\mu.
	\end{align*}
	By
	$$
	|\zeta-\zeta_{2}  | \leq | \zeta_{2}-\zeta_{1}  |\leq | \zeta_{1} |,
	$$
	and 
	\begin{align*}
	| \zeta |^{s+1}(1+| \zeta_{1} |)^{-s}& = | \zeta |^{s+1} | \zeta_{1} |^{-s} | \zeta_{1} |^{s} (1+| \zeta_{1} |)^{-s}
	\leq | \zeta |^{s+1} | \zeta_{1} |^{-s}  \frac{|\zeta_{1} |^{s}}{(1+| \zeta_{1} |)^{s} } \leq | \zeta |^{s+1} | \zeta_{1} |^{-s},
	\end{align*}
	and
	\begin{align*}
	  | \zeta |^{s+1} ~ | \zeta_{1}  |^{-s}& = | \zeta |^{\frac{1}{2}} | \zeta_{1}  |^{\frac{1}{2}} | \zeta |^{s+ \frac{1}{2} }| \zeta_{1}  |^{-s- \frac{1}{2}} \\ &\leq c   | \zeta |^{\frac{1}{2}} | \zeta_{1}  |^{\frac{1}{2}}  \left( | \zeta-\zeta_{2}|)^{s+ \frac{1}{2} } +   | \zeta_{2}-\zeta_{1}|^{s+ \frac{1}{2} } + | \zeta_{1}|^{s+ \frac{1}{2}} \right)| \zeta_{1}  |^{-s- \frac{1}{2}} \\
	&\leq c  | \zeta |^{\frac{1}{2}} | \zeta_{1}  |^{\frac{1}{2}}   \left( |\zeta_{1}|)^{s+ \frac{1}{2} } +   | \zeta_{1}|^{s+ \frac{1}{2} } + | \zeta_{1}|^{s+ \frac{1}{2}} \right)| \zeta_{1}  |^{-s- \frac{1}{2}} \\
	& \leq  c  | \zeta |^{\frac{1}{2}} | \zeta_{1}  |^{\frac{1}{2}} \left( 3 |\zeta_{1}  |^{s+\frac{1}{2}}~~| \zeta_{1}  |^{-s- \frac{1}{2}} \right) \\ &
\leq C | \zeta |^{\frac{1}{2}} | \zeta_{1}  |^{\frac{1}{2}}.
	\end{align*}
We suppose that
 \begin{align*}
\widehat{A^{\frac{1}{2}} M_{-b'}}(\zeta,\eta)& = \frac{( | \zeta |)^{\frac{1}{2}}  m(\zeta,\eta)} {(1+  |\eta - \zeta^{3} |)^{-b'}}\\
 \widehat{A^{\frac{1}{2}} F_{b}}(\zeta_{1},\eta_{1})  &=  \frac{(1+  | \zeta_{1} |)^{\frac{1}{2}} f_{1} (\zeta_{1},\eta_{1})} {(1+  |\eta_{1} - \zeta_{1}^{3} |)^{b}}\\
\widehat{A^{-s} G^{1}_{b}}(\zeta-\zeta_{2},\eta-\eta_{2}) &=\frac{(1+  | \zeta-\zeta_{2}  |)^{-s} g_{1}(\zeta-\zeta_{2},\eta-\eta_{2})} {(1+  |\eta -\eta_{2} - (\zeta-\zeta_{2})^{3} |)^{b}}\\
\widehat{A^{-s} G^{2}_{b}}(\zeta_{2}-\zeta_{1},\eta_{2}-\eta_{1})&= \frac{(1+  | \zeta_{2}-\zeta_{1}  |)^{-s} g_{2}(\zeta_{2}-\zeta_{1},\eta_{2}-\eta_{1})} {(1+  |\eta_{2}-\eta_{1} - (\zeta_{2}-\zeta_{1})^{3} |)^{b}},
\end{align*}
and
\begin{align*}
 &\int_{\mathbb{R}^{6}}\frac{( | \zeta |)^{\frac{1}{2}}  m(\zeta,\eta)} {(1+  |\eta - \zeta^{3} |)^{-b'}}~~~~
 \frac{(1+  | \zeta_{1} |)^{\frac{1}{2}} f_{1} (\zeta_{1},\eta_{1})} {(1+  |\eta_{1} - \zeta_{1}^{3} |)^{b}}~
 \frac{(1+  | \zeta-\zeta_{2}  |)^{-s} g_{1}(\zeta-\zeta_{2},\eta-\eta_{2})} {(1+  |\eta -\eta_{2} - (\zeta-\zeta_{2})^{3} |)^{b}}
  \frac{(1+  | \zeta_{2}-\zeta_{1}  |)^{-s} g_{2}(\zeta_{2}-\zeta_{1},\eta_{2}-\eta_{1})} {(1+  |\eta_{2}-\eta_{1} - (\zeta_{2}-\zeta_{1})^{3} |)^{b}} d\mu\\&
  \int_{\mathbb{R}^{6}} \widehat{A^{\frac{1}{2}} M_{-b'}}(\zeta,\eta)   \widehat{A^{\frac{1}{2}} F_{b}}(\zeta_{1},\eta_{1})
\widehat{A^{-s} G^{1}_{b}}(\zeta-\zeta_{2},\eta-\eta_{2})
\widehat{A^{-s} G^{2}_{b}}(\zeta_{2}-\zeta_{1},\eta_{2}-\eta_{1})
         d\mu\\&
= \int_{\mathbb{R}^{2}}  \left(  \widehat{A^{\frac{1}{2}} M_{-b'}}(\zeta,\eta) \right)  \left(  \int_{\mathbb{R}^{4}}  \widehat{A^{\frac{1}{2}} F_{b}}(\zeta_{1},\eta_{1})
\widehat{A^{-s} G^{1}_{b}}(\zeta-\zeta_{2},\eta-\eta_{2})
\widehat{A^{-s} G^{2}_{b}}(\zeta_{2}-\zeta_{1},\eta_{2}-\eta_{1}) d\zeta_{1} d\eta_{1} d\zeta_{2} d\eta_{2} \right)
         d\zeta d\eta
\\&=    \int_{\mathbb{R}^{2}}  \left(  \widehat{A^{\frac{1}{2}} M_{-b'}}(\zeta,\eta) \right)  \left(  \left(  \widehat{A^{\frac{1}{2}} F_{b}}\ast
\widehat{A^{-s} G^{1}_{b}}\ast
\widehat{A^{-s} G^{2}_{b}} \right) (\zeta,\eta) \right)
         d\zeta d\eta
 \\& =   \int_{\mathbb{R}^{2}}  \left(  \widehat{A^{\frac{1}{2}} M_{-b'}}(\zeta,\eta) \right)  \left(  \left(  \widehat{A^{\frac{1}{2}} F_{b}. A^{-s} G^{1}_{b}. A^{-s} G^{2}_{b}}\right) (\zeta,\eta) \right)
         d\zeta d\eta
 \\& =   \int_{\mathbb{R}^{2}}  A^{\frac{1}{2}}  M_{-b'}(x,t) \left(  A^{\frac{1}{2}} F_{b}. A^{-s} G^{1}_{b}. A^{-s} G^{2}_{b} \right) (x,t)  dx dt.
\end{align*}
We suppose that
\begin{align*}
h_{1}(x,t) &= A^{\frac{1}{2}}  M_{-b'}(x,t) \\
h_{2}(x,t) &= A^{\frac{1}{2}} F_{b}(x,t)\\
h_{3}(x,t) &= A^{-s} G^{1}_{b} (x,t)\\
h_{4}(x,t) &=A^{-s} G^{2}_{b}(x,t),
\end{align*} 
 then
 \begin{align*}
  \bigg|\int_{\mathbb{R}^{2}}  A^{\frac{1}{2}}  D_{-b'}(x,t)  A^{\frac{1}{2}} F_{b}. A^{-s} G^{1}_{b}. A^{-s} G^{2}_{b} (x,t)  dx dt\bigg| &= \bigg|\int_{\mathbb{R}^{2}} h_{1}(x,t). h_{2}(x,t)  h_{3}(x,t) h_{4}(x,t) dx dt\bigg|\\&
  \leq \bigg| \int_{\mathbb{R}^{2}} h_{1}(x,t). h_{2}(x,t) \sup_{\substack{ t \in [0, T]}} h_{3}(x,t) \sup_{\substack{ t \in [0, T]
  }} h_{4}(x,t) dx dt\bigg|\\&
\leq \bigg| \int_{\mathbb{R}^{2}}\left(  h_{1}(x,t). h_{2}(x,t)  \right)  \left(  \sup_{\substack{ t \in [0, T]}} h_{3}(x,t) \sup_{\substack{ t \in [0, T]
  }} h_{4}(x,t) \right)  dx dt\bigg|.&
 \end{align*}
By using Cauchy-Schwarz's inequality  for the variables $  x $ and $ t  $
 \begin{align*}
  & \bigg|\int_{\mathbb{R}^{2}}\left(  h_{1}(x,t). h_{2}(x,t)  \right)  \left(  \sup_{\substack{ t \in [0, T]}} h_{3}(x,t) \sup_{\substack{ t \in [0, T]
  }} h_{4}(x,t) \right)  dx dt\bigg| \\ &
\leq
     \| h_{1}(x,t) \|_{L_{x}^{4} L_{t}^{2} }   \|  h_{2}(x,t) \|_{L_{x}^{4} L_{t}^{2} }  \| h_{3}(x,t) \|_{L_{x}^{2} L_{t}^{\infty}}  \| h_{4}(x,t) \|_{L_{x}^{\infty} L_{t}^{\infty} }
  \\&
 =  \| A^{\frac{1}{2}} M_{-b'} \|_{L_{x}^{4} L_{t}^{2} }   \| A^{\frac{1}{2}} F_{b} \|_{L_{x}^{4} L_{t}^{2} }  \| A^{-s}
   G^{1}_{b} \|_{L_{x}^{2} L_{t}^{\infty}}  \| A^{\frac{1}{2}} G^{2}_{b} \|_{L_{x}^{\infty} L_{t}^{\infty} }.
 \end{align*}
Then 	
	\begin{equation*}
	 	 \| \partial_{x} u_{1}  v_{1} v_{2}  \|_{\rho, s, b'}
	\leq  c  \| A^{\frac{1}{2}} M_{-b'} \|_{L_{x}^{4} L_{t}^{2} }   \| A^{\frac{1}{2}} F_{b} \|_{L_{x}^{4} L_{t}^{2} }  \| A^{-s}
	G^{1}_{b} \|_{L_{x}^{2} L_{t}^{\infty}}  \| A^{\frac{1}{2}} G^{2}_{b} \|_{L_{x}^{\infty} L_{t}^{\infty} }.
	\end{equation*}
	Hence by Lemma \ref{2.3}
\begin{align*}
 \| \partial_{x} u_{1}  v_{1} v_{2}  \|_{\rho, s, b'} & \leq  c  \| m \|_{L_{\zeta}^{2} L_{\eta}^{2} }   \| f \|_{L_{\zeta}^{2} L_{\eta}^{2} }  \| g_{1} \|_{L_{\zeta}^{2} L_{\eta}^{2} }   \| g_{2} \|_{L_{\zeta}^{2} L_{\eta}^{2} } \\ &
\leq c\| u_{1}\|_{\rho, s, b}  \| v_{1} \|_{\rho, s, b}  \| v_{2} \|_{\rho, s, b}.
\end{align*}
	Secondly for  the case $  |\zeta| \leq 1 $, then
	\begin{align*}
(1+  | \zeta |)^{s} | \zeta | (1+  | \zeta_{1} |)^{-s} &= 	(1+  | \zeta |)^{^{\frac{1}{2}}} (1+  | \zeta_{1} |)^{^{\frac{1}{2}}}
 (1+  | \zeta_{1} |)^{-s-\frac{1}{2}}(1+  | \zeta |)^{s-\frac{1}{2}}   | \zeta |  \\
&\leq 	(1+  | \zeta |)^{^{\frac{1}{2}}} (1+  | \zeta_{1} |)^{^{\frac{1}{2}}}
1+  | \zeta_{1} |)^{-s-\frac{1}{2}}(1+  | \zeta |)^{s-\frac{1}{2}}   ( 1+ |\zeta | )
 \\
&\leq 	(1+  | \zeta |)^{^{\frac{1}{2}}} (1+  | \zeta_{1} |)^{^{\frac{1}{2}}}(1+  | \zeta_{1} |)^{-s-\frac{1}{2}}
(1+  | \zeta |)^{s+\frac{1}{2}}
\\
 & \leq (1+  | \zeta |)^{^{\frac{1}{2}}} (1+  | \zeta_{1} |)^{^{\frac{1}{2}}}
(1+  | \zeta_{1} |)^{-s-\frac{1}{2}}  (1+  | \zeta_{1}  |  +  | \zeta - \zeta_{2}  | + | \zeta_{2}  - \zeta_{1} |  )^{s+\frac{1}{2}}
 \\& \leq (1+  | \zeta |)^{^{\frac{1}{2}}} (1+  | \zeta_{1} |)^{^{\frac{1}{2}}}
(1+  | \zeta_{1} |)^{-s-\frac{1}{2}}  ( 3(1+  | \zeta_{1}  | )   )^{s+\frac{1}{2}}
 \\
 & \leq  C   (1+  | \zeta |)^{^{\frac{1}{2}}} (1+  | \zeta_{1} |)^{^{\frac{1}{2}}},
	\end{align*}
	then
	\begin{align*}	
	 	 \| \partial_{x} u_{1}  v_{1} v_{2}  \|_{\rho, s, b'} & \leqslant \int_{\mathbb{R}^{6}}\frac{(1+  | \zeta |)^{s} | \zeta | m(\zeta,\eta)} {(1+  |\eta - \zeta^{3} |)^{-b'}}~
	\frac{(1+  | \zeta_{1} |)^{-s} f_{1} (\zeta_{1},\eta_{1})} {(1+  |\eta_{1} - \zeta_{1}^{3} |)^{b}}~
	\frac{(1+  | \zeta-\zeta_{2}  |)^{-s} g_{1}(\zeta-\zeta_{2},\eta-\eta_{2})} {(1+  |\eta -\eta_{2} - (\zeta-\zeta_{2})^{3} |)^{b}}~\\& \quad\times
	\frac{(1+  | \zeta_{2}-\zeta_{1}  |)^{-s} g_{2}(\zeta_{2}-\zeta_{1},\eta_{2}-\eta_{1})} {(1+  |\eta_{2}-\eta_{1} - (\zeta_{2}-\zeta_{1})^{3} |)^{b}}d\mu
	\\&
	\leq C
	\int_{\mathbb{R}^{6}}\frac{  (1+ | \zeta  |)^{\frac{1}{2}} m(\zeta,\eta)} {(1+  |\eta - \zeta^{3} |)^{-b'}}~
	\frac{(1+  | \zeta_{1} |)^{\frac{1}{2}} f_{1} (\zeta_{1},\eta_{1})} {(1+  |\eta_{1} - \zeta_{1}^{3} |)^{b}}~
	\frac{(1+  | \zeta-\zeta_{2}  |)^{-s} g_{1}(\zeta-\zeta_{2},\eta-\eta_{2})} {(1+  |\eta -\eta_{2} - (\zeta-\zeta_{2})^{3} |)^{b}}~
	\\& \quad\times
	\frac{(1+  | \zeta_{2}-\zeta_{1}  |)^{-s} g_{2}(\zeta_{2}-\zeta_{1},\eta_{2}-\eta_{1})} {(1+  |\eta_{2}-\eta_{1} - (\zeta_{2}-\zeta_{1})^{3} |)^{b}} d\mu.
	\end{align*}
	Then, by the inner product, we have
	 \begin{align*}
 \| \partial_{x} u_{1}  v_{1} v_{2}  \|_{\rho, s, b'} & \leq c  \langle \widehat{A^{\frac{1}{2}} M_{-b'}}; \widehat{A^{\frac{1}{2}} F_{b}}\star \widehat{A^{-s} G^{1}_{b}} \star \widehat{A^{-s} G^{2}_{b}}\rangle   \\ &
\leq  c  \langle \widehat{A^{\frac{1}{2}} M_{-b'} }; \widehat{  A^{\frac{1}{2}} F_{b}~.  A^{-s} G^{1}_{b}~A^{-s}G^{2}_{b}}  \rangle \\ &
 \leq c\langle A^{\frac{1}{2}} M_{-b'};  A^{\frac{1}{2}} F_{b}. A^{-s} G^{1}_{b}A^{-s} G^{2}_{b}  \rangle   \\ &
\leq  c  \| A^{\frac{1}{2}} M_{-b'} \|_{L_{x}^{4} L_{t}^{2} }   \| A^{\frac{1}{2}} F_{b} \|_{L_{x}^{4} L_{t}^{2} }  \| A^{-s}
G^{1}_{b} \|_{L_{x}^{2} L_{t}^{\infty}}  \| A^{-s} G^{2}_{b} \|_{L_{x}^{\infty} L_{t}^{\infty} }.
	 \end{align*}
		Hence by Lemma \ref{2.3}
		\begin{align*}
	 	 \| \partial_{x} u_{1}  v_{1} v_{2}  \|_{\rho, s, b'} & \leq  c  \| m \|_{L_{\zeta}^{2} L_{\eta}^{2} }   \| f \|_{L_{\zeta}^{2} L_{\eta}^{2} }  \| g_{1} \|_{L_{\zeta}^{2} L_{\eta}^{2} }   \| g_{2} \|_{L_{\zeta}^{2} L_{\eta}^{2} } \\ &
		 \leq  c    \| u_{1}\|_{\rho, s, b}  \| v_{1} \|_{\rho, s, b}  \| v_{2} \|_{\rho, s, b}.
		 \end{align*}
		By the same way, we prove the inequality in the five region.\\
		 For the case $  p \geq  2 $ is   virtually identical. The only difference is that  we need to split the Fourier space in  $((2p+1)+1)!$.\\
		We prove that
		\begin{align*}
		\| \partial_{x}\prod _{i=1}^{p} u_{i} \prod _{j=1}^{p+1} v_{j}  \|_{\rho, s, b'}\leq
		C  \prod _{i=1}^{p}\|  u_{i}  \|_{\rho, s, b}. \prod _{j=1}^{p+1}\|  v_{j}  \|_{\rho, s, b},
		\end{align*}
		We have :
		\begin{align*}
		\| \partial_{x}\prod _{i=1}^{p} u_{i} \prod _{j=1}^{p+1} v_{j}  \|_{\rho, s, b'} & =  \|  (1+  | \zeta |)^{s} (1+  |\eta - \zeta^{3} |)^{b'} e^{\rho (1+| \zeta |)} |     \widehat{ \partial_{x}\prod _{i=1}^{p} u_{i} \prod _{j=1}^{p+1} v_{j} }(\zeta, \eta  ) \|_{L^{2}_{\zeta} L^{2}_{\eta} },
\\ &=  \|  (1+  | \zeta |)^{s} (1+  |\eta - \zeta^{3} |)^{b'} e^{\rho (1+| \zeta |)}
		|     \zeta \prod _{i=1}^{p} \widehat{ u_{i} }\star \prod _{j=1}^{p+1}  \widehat{v_{j} } (\zeta, \eta  ) \|_{L^{2}_{\zeta} L^{2}_{\eta} }.
		\end{align*}
		By the same way,  by the inner product, we have
		\begin{align*}
		\| \partial_{x}\prod _{i=1}^{p} u_{i} \prod _{j=1}^{p+1} v_{j}  \|_{\rho, s, b'}	& \leq c  \langle \widehat{A^{\frac{1}{2}} M_{-b'}}; \widehat{A^{\frac{1}{2}} F_{b}}\star  \widehat{A^{-s} G^{1}_{b}}  \star  \prod _{i=1}^{p} \widehat{A^{-s} F^{i}_{b}} \star \prod _{j=1}^{p}\widehat{A^{-s} G^{i}_{b}}\rangle \\ &
		\leq c  \langle \widehat{A^{\frac{1}{2}} M_{-b'}}; \widehat{A^{\frac{1}{2}} F_{b}. A^{-s} G^{1}_{b}  \prod _{i=1}^{p} A^{-s} F^{i}_{b}\prod _{j=1}^{p}A^{-s} G^{i}_{b}}\rangle \\ &
		\leq  c  \| A^{\frac{1}{2}} M_{-b'} \|_{L_{x}^{4} L_{t}^{2} }   \| A^{\frac{1}{2}} F_{b} \|_{L_{x}^{4} L_{t}^{2} } \| A^{-s} G^{1}_{b}  \|_{L_{x}^{2} L_{t}^{\infty}}  \| \prod _{i=1}^{p}A^{-s}
		F^{i}_{b} \|_{L_{x}^{\infty} L_{t}^{\infty}}  \| \prod _{j=1}^{p} A^{\frac{1}{2}} G^{i}_{b} \|_{L_{x}^{\infty} L_{t}^{\infty} } \\ &
			\leq  c  \prod _{i=1}^{p}\|  u_{i}  \|_{\rho, s, b}. \prod _{j=1}^{p+1}\|  v_{j}  \|_{\rho, s, b}.
		\end{align*}

	\end{proof}
\begin{lemma}
Let  $\rho > 0 $, $  s \geq 3b $,  $ b >\frac{1}{2}  $, and $ b ' < - \frac{1}{4} $. Let $ p \in \mathbb{N}$ and suppose that $ u_{1},..., u_{p+1} $\\ $, v_{1},...,v_{p+1}  \in X_{\rho, s, b}$. Then there exists a constants $ c $   such that
	\begin{equation*}
	\| \partial_{x}\prod _{i=1}^{p} u_{i} \prod _{j=1}^{p+1} v_{j}  \|_{\rho, s, b'}\leq
	C  \prod _{i=1}^{p}\|  u_{i}  \|_{  s, b}. \prod _{j=1}^{p+1}\|  v_{j}  \|_{ s, b} + c   \prod _{i=1}^{p}\|  u_{i}  \|_{ \rho,  s, b}. \prod _{j=1}^{p+1}\|  v_{j}  \|_{\rho, s, b},
	\end{equation*}
	\begin{equation*}
	\| \partial_{x}\prod _{i=1}^{p+1} u_{i} \prod _{j=1}^{p} v_{j}  \|_{\rho, s, b'}\leq
	C  \prod _{i=1}^{p+1}\|  u_{i}  \|_{  s, b}. \prod _{j=1}^{p}\|  v_{j}  \|_{ s, b} + c   \prod _{i=1}^{p+1}\|  u_{i}  \|_{ \rho,  s, b}. \prod _{j=1}^{p}\|  v_{j}  \|_{\rho, s, b}.
	\end{equation*}
	
\end{lemma}

\begin{proof}
	We begin by the case $   p = 1 $, thats mean we prove that 
	\begin{equation}
	\| \partial_{x} ( u_{1}  v_{1}  v_{2}  ) \|_{\rho, s, b'}\leq
	C \|  u_{1}  \|_{  s, b}. \|  v_{1}  \|_{ s, b} \|  v_{2}  \|_{ s, b} + c  \|  u_{1}  \|_{ \rho,  s, b}. \|  v_{1}  \|_{\rho, s, b}  \|  v_{2}  \|_{\rho, s, b}.
	\end{equation}
	We define
	$$
	f_{i} (\zeta, \eta ) = (1+  | \zeta |)^{s} (1+  |\eta - \zeta^{3} |)^{b} e^{\rho (1+  | \zeta |)} | \widehat{u_{i} }(\zeta,\eta )|, $$
	$$
	g_{j} (\zeta, \eta ) = (1+  | \zeta |)^{s} (1+  |\eta - \zeta^{3} |)^{b} e^{\rho (1+  | \zeta |)  }| \widehat{v_{j} }(\zeta,\eta )|. $$
	Then
	\begin{align*}
	\| \partial_{x} u_{1}  v_{1} v_{2}  \|_{\rho, s, b'}&=  \|  (1+  | \zeta |)^{s} (1+  |\eta - \zeta^{3} |)^{b'} e^{\rho (1+| \zeta |)} | \widehat{\partial_{x} u_{1}  v_{1} v_{2} }(\zeta,\eta )|     \|_{L^{2}_{\zeta} L^{2}_{\eta} }  \\ \\ &
	= \|  (1+  | \zeta |)^{s}e^{\rho (1+| \zeta |)} (1+  |\eta - \zeta^{3} |)^{b'}  ~  | \zeta |~~  |\widehat{u_{1}}\ast \widehat{v_{1}} \ast \widehat{v_{2}}(\zeta,\eta )|     \|_{_{L^{2}_{\zeta} L^{2}_{\eta} }}
	\\ \\&
	= \|  (1+  | \zeta |)^{s}e^{\rho (1+| \zeta |)} (1+  |\eta - \zeta^{3} |)^{b'}  ~  | \zeta | ~ \int_{\mathbb{R}^{4}}  \widehat{u_{1}}(\zeta_{1},\eta_{1} ) \widehat{v_{1}}(\zeta-\zeta_{2},\eta - \eta_{2} )
	\\ \\&\quad \times \widehat{v_{2}}(\zeta_{2}-\zeta_{1},\eta_{2}-\eta_{1} )| d\zeta_{1} d\eta_{1} d\zeta_{2} d\eta_{2}   \|_{_{L^{2}_{\zeta} L^{2}_{\eta} }}
	\\ \\ &  = \|  (1+  | \zeta |)^{s} e^{\rho (1+| \zeta |)}(1+  |\eta - \zeta^{3} |)^{b'}   | \zeta |  \int_{\mathbb{R}^{4}} ~~
	\frac{(1+  | \zeta_{1} |)^{-s}  e^{-\rho (1+  | \zeta_{1} |)} \widehat{f_{1}}(\zeta_{1},\eta_{1} )}{(1+  | \eta - \zeta^{3} |)^{b}} \\ \\& \quad \times
	\frac{(1+  | \zeta-\zeta_{2}  |)^{-s}  e^{-\rho (1+  |\zeta-\zeta_{2} |)} \widehat{g_{1}}(\zeta-\zeta_{2},\eta - \eta_{2})}{(1+  | ( \eta - \eta_{2}  ) - (\zeta-\zeta_{2})^{3} |)^{b}} \\ \\& \quad \times
	(\frac{(1+  |  \zeta_{2} -\zeta_{1} |)^{-s}  e^{-\rho (1+  | \zeta_{2} -\zeta_{1}|)} \widehat{g_{2}}(\zeta_{2} -\zeta_{1},\eta_{2}-\eta_{1} )}{(1+  |\eta_{2}-\eta_{1}  -(\zeta_{2} -\zeta_{1},\eta_{2}-\eta_{1} )^{3} |)^{b}}   d \mu \|_{{L^{2}_{\zeta} L^{2}_{\eta} }}.
	\end{align*}
We proof this estimate by the duality. Let $  m(\zeta,\eta) $ be a positive function in $  L^{2}( \mathbb{R}^{2})$ with norm  $ \| m \|_{ L^{2}( \mathbb{R}^{2})}=1   $, then 
	\begin{align*}
	&\int_{\mathbb{R}^{6}}\frac{e^{\rho (1+| \zeta |)}(1+  | \zeta |)^{s} | \zeta | m(\zeta,\eta)} {(1+  |\eta - \zeta^{3} |)^{-b'}}
	\frac{e^{-\rho (1+  | \zeta_{1} |)}(1+  | \zeta_{1} |)^{-s} f_{1} (\zeta_{1},\eta_{1})} {(1+  |\eta_{1} - \zeta_{1}^{3} |)^{b}} \\ \\ &
	\frac{e^{-\rho (1+  | \zeta-\zeta_{2}   |)}(1+  | \zeta-\zeta_{2}  |)^{-s} g_{1}(\zeta-\zeta_{2},\eta-\eta_{2})} {(1+  |\eta -\eta_{2} - (\zeta-\zeta_{2})^{3} |)^{b}}~~~~
	\frac{e^{-\rho (1+  | \zeta_{2}-\zeta_{1}   |)}(1+  | \zeta_{2}-\zeta_{1}  |)^{-s} g_{2}(\zeta_{2}-\zeta_{1},\eta_{2}-\eta_{1})} {(1+  |\eta_{2}-\eta_{1} - (\zeta_{2}-\zeta_{1})^{3} |)^{b}}d\mu.
	\end{align*}
	Using the inequality
	\begin{equation*}
	e^{\rho (1+| \zeta |)}  \leq  e +  \rho^{\frac{1}{2}}  e^{\rho (1+  | \zeta  |)}    (1+  | \zeta  |)^{\frac{1}{2}}.
	\end{equation*}
	Then
	\begin{align*}
	&\int_{\mathbb{R}^{6}}\frac{ e^{\rho (1+  | \zeta  |)}  (1+  | \zeta |)^{1+s}  m(\zeta,\eta)} {(1+  |\eta - \zeta^{3} |)^{-b'}}
	\frac{  e^{-\rho (1+  |\zeta_{1} |)}  (1+  | \zeta_{1} |)^{-s} f_{1} (\zeta_{1},\eta_{1})} {(1+  |\eta_{1} - \zeta_{1}^{3} |)^{b}}\\ \\&
	\frac{  e^{-\rho (1+  |\zeta-\zeta_{2} |)}  (1+  | \zeta-\zeta_{2}  |)^{-s} g_{1}(\zeta-\zeta_{2},\eta-\eta_{2})} {(1+  |\eta -\eta_{2} - (\zeta-\zeta_{2})^{3} |)^{b}}\frac{  e^{-\rho (1+  |\zeta_{2}-\zeta_{1} |)}   (1+  | \zeta_{2}-\zeta_{1}  |)^{-s} g_{2}(\zeta_{2}-\zeta_{1},\eta_{2}-\eta_{1})} {(1+  |\eta_{2}-\eta_{1}-(\zeta_{2}-\zeta_{1})^{3} |)^{b}}\\ \\&
	 \leq I +I^{\prime},
	\end{align*}
	where
	\begin{align*}
	I + I^{\prime}& = e \sup_{m \in B}
	\int_{\mathbb{R}^{6}}\frac{   (1+  | \zeta |)^{1+s}  m(\zeta,\eta)} {(1+  |\eta - \zeta^{3} |)^{-b'}}~~~~
	\frac{  e^{-\rho (1+  |\zeta_{1} |)}  (1+  | \zeta_{1} |)^{-s} f_{1} (\zeta_{1},\eta_{1})} {(1+  |\eta_{1} - \zeta_{1}^{3} |)^{b}}\\ \\& \quad \times
	\frac{  e^{-\rho (1+  |\zeta-\zeta_{2} |)}  (1+  | \zeta-\zeta_{2}  |)^{-s} g_{1}(\zeta-\zeta_{2},\eta-\eta_{2})} {(1+  |\eta -\eta_{2} - (\zeta-\zeta_{2})^{3} |)^{b}}
	\frac{  e^{-\rho (1+  |\zeta_{2}-\zeta_{1} |)}   (1+  | \zeta_{2}-\zeta_{1}  |)^{-s} g_{2}(\zeta_{2}-\zeta_{1},\eta_{2}-\eta_{1})} {(1+  |\eta_{2}-\eta_{1} - (\zeta_{2}-\zeta_{1})^{3} |)^{b}}d\mu
	\\ \\& \quad +  \rho^{\frac{1}{2}} \sup_{m \in B}
	\int_{\mathbb{R}^{6}}\frac{  e^{\rho (1+  | \zeta  |)}    (1+  | \zeta  |)^{\frac{1}{2}}   (1+  | \zeta |)^{1+s}  m(\zeta,\eta)} {(1+  |\eta - \zeta^{3} |)^{-b'}}~~    \frac{  e^{-\rho (1+  |\zeta_{1} |)}  (1+  | \zeta_{1} |)^{-s} f_{1} (\zeta_{1},\eta_{1})} {(1+  |\eta_{1} - \zeta_{1}^{3} |)^{b}}\\ \\& \quad \times
	\frac{  e^{-\rho (1+  |\zeta-\zeta_{2} |)}  (1+  | \zeta-\zeta_{2}  |)^{-s} g_{1}(\zeta-\zeta_{2},\eta-\eta_{2})} {(1+  |\eta -\eta_{2} - (\zeta-\zeta_{2})^{3} |)^{b}}
	\frac{  e^{-\rho (1+  |\zeta_{2}-\zeta_{1} |)}   (1+  | \zeta_{2}-\zeta_{1}  |)^{-s} g_{2}(\zeta_{2}-\zeta_{1},\eta_{2}-\eta_{1})} {(1+  |\eta_{2}-\eta_{1} - (\zeta_{2}-\zeta_{1})^{3} |)^{b}} d\mu.
	\end{align*}
	Now, split the Fourier space into six regions ( the same division as before $ \ref{ta14}  $). We begin by the case $ (1) $
	$
	\left( |\zeta-\zeta_{2}  | \leq | \zeta_{2}-\zeta_{1}  |\leq | \zeta_{1} | \right).$	The integrale of
	$ I $ corresponding to the particular region just delineated can be dominated by the supremum over all $m$ in $B$ of the duality relation the integral can be dominated  by the inner product.
	\begin{equation*}
	\begin{array}{l}
  I \leq c  \langle \widehat{A^{\frac{1}{2}} M_{-b'}}; \widehat{  e^{-\rho A } A F_{b}}\star  \widehat{  e^{-\rho A } A^{-s} G^{1}_{b}}   \star \widehat{ e^{-\rho A}  A^{-s} G^{2}_{b}}\rangle \\ \\
	\leq c  \langle \widehat{A^{\frac{1}{2}} M_{-b'}}; \widehat{  e^{-\rho  A} A F_{b}~. ~  e^{-\rho A} A^{-s} G^{1}_{b}  ~. ~  e^{-\rho A }  A^{-s} G^{2}_{b}}\rangle \\ \\
		\leq c  \langle A^{\frac{1}{2}} M_{-b'};   e^{-\rho A } A F_{b} ~. ~   e^{-\rho  A} A^{-s} G^{1}_{b} ~.~  e^{-\rho A }  A^{-s} G^{2}_{b}\rangle  \\ \\
		\leq  c  \| A^{\frac{1}{2}} M_{-b'} \|_{L_{x}^{4} L_{t}^{2} }   \| e^{-\rho A }  A F_{b} \|_{L_{x}^{\infty} L_{t}^{2} }  \| e^{-\rho A } A^{-s}
		G^{1}_{b} \|_{L_{x}^{2} L_{t}^{\infty}}  \|  e^{-\rho A } A^{-s} G^{2}_{b} \|_{L_{x}^{4} L_{t}^{\infty} }.
	\end{array}
	\end{equation*}
	Hence by Lemma \ref{2.3}
	\begin{equation*}
	\begin{array}{l}
	 I \leq  c  \| m \|_{L_{\zeta}^{2} L_{\eta}^{2} }   \| e^{-\rho A } f_{1} \|_{L_{\zeta}^{2} L_{\eta}^{2} }  \| e^{-\rho A } g_{1} \|_{L_{\zeta}^{2} L_{\eta}^{2} }   \|  e^{-\rho A } g_{2} \|_{L_{\zeta}^{2} L_{\eta}^{2} }
	\leq  c    \| u_{1}\|_{ s, b}  \| v_{1} \|_{ s, b}  \| v_{2} \|_{ s, b}.
	\end{array}
	\end{equation*}	
	By the same way, we treat the second part, that is, the integration  $ I^{\prime} $ and we use the following inequality
	$$  e^{\rho (1+| \zeta |)}  \leq  e^{\rho (1+  | \zeta_{1} |)}  \times e^{\rho (1+  | \zeta-\zeta_{2}   |)}   \times e^{\rho (1+  | \zeta_{2}-\zeta_{1}   |)},
	$$
	we find 
	\begin{equation*}
	\begin{array}{l}
	 I^{\prime}   \leq  c \rho^{\frac{1}{2}}\sup_{\substack{m \in B }}\| A^{\frac{1}{2}} M_{-b'} \|_{L_{x}^{4} L_{t}^{2} }   \|   A^{\frac{1}{2}} F_{b} \|_{L_{x}^{4} L_{t}^{2} }  \|  A^{-s}
	G^{1}_{b} \|_{L_{x}^{2} L_{t}^{\infty}}  \|   A^{-s} G^{2}_{b} \|_{L_{x}^{\infty} L_{t}^{\infty} }  \\ \\
	\leq  c  \| m\|_{L_{\zeta}^{2} L_{\eta}^{2} }   \|  f_{1} \|_{L_{\zeta}^{2} L_{\eta}^{2} }  \|  g_{1} \|_{L_{\zeta}^{2} L_{\eta}^{2} }   \|  g_{2} \|_{L_{\zeta}^{2} L_{\eta}^{2} } \\ \\
	\leq  c    \| u_{1}\|_{\rho, s, b}  \| v_{1} \|_{\rho, s, b}  \| v_{2} \|_{\rho, s, b}.
	\end{array}
	\end{equation*}	
	The other five cases, follow by symmetry.\\
	For the case $  p > 2 $, the same scheme of estimation will yield  for$( p -2) $ with additional factors of the form
	$$  \|   A^{-s}(G_{b}^{i} )  A^{-s}(F_{b}^{i} ) \|_{L_{x}^{\infty}  L_{t}^{\infty}}.     $$
	We deal with the rest of the parts in the same way
\end{proof}

\section{Proof of Theorem \ref{the1.2}}

{\bf Existence of solution.} We define 
\begin{equation*}
\mathcal{B}_{\rho, s, b} = X_{\rho, s, b} \times X_{\rho, s, b}, \quad \quad
\mathcal{N}^{\rho, s}= \mathcal{G}_{\rho, s} \times \mathcal{G}_{\rho, s},
\end{equation*}
\begin{equation*}
\| (u, v) \|_{\mathcal{B}_{\rho, s, b}} =  \max \{\| u \|_{\rho, s, b};  \| u \|_{\rho, s, b} \} \quad  \text{and} \quad \| (u_0, v_0) \|_{\mathcal{N}^{\rho, s}} =  \max \{\|u_0 \|_{\mathcal{G}_{\rho, s}};  \| v_0 \|_{\mathcal{G}_{\rho, s}} \}.
\end{equation*}

\begin{lemma}
	Let $ s\geq 0 $,$  \rho \geq 0 $, $ b > \frac{1}{2}  $ and $   T \in ( 0; 1)$.
	Then, for all $ (u_{0}, v_{0}) \in \mathcal{N}^{\rho, s}  $, the map \\
	$  \varXi \times \Gamma: B(0,R) \longrightarrow B(0,R) $ is a contraction, where  $   B (0, R)$ is given by 
	\begin{equation*}
	\mathbb{B}(0,R) = \{  (u,v) \in \mathcal{B}_{\rho, s, b};\quad \| (u, v) \|_{\mathcal{B}_{\rho, s, b}} \leq R \} \quad \text{where} \quad
	R= 2C\| (u_{0}, v_{0}) \|_{\mathcal{N}^{\rho, s}}.
	\end{equation*}
	
\end{lemma}
\begin{proof}
	First it is proved that $  \varXi \times \Gamma $
	is mapping on $ \mathbb{B}(0,R) $ 
	\begin{align*}
	\|  \varXi [u,v](t) \|_{\rho, s, b}
	&=    \| \psi (t) W(t)u_{0}- \psi_{T}(t)\int_{0}^{t} W(t-t^{\prime})w_{1}(t^{\prime}) dt^{\prime} \|_{\rho, s, b}  \\
	&\leq   \| \psi (t) W(t)u_{0} \|_{\rho, s, b}+ \|  \psi_{T}(t)\int_{0}^{t} W(t-t^{\prime})w_{1}(t^{\prime}) dt^{\prime}  \|_{\rho, s, b}\\
	&\leq C \| u_{0} \|_{\mathcal{G}_{\rho, s}} + CT^{1-b+b'}
	\|  w_{1}(t^{\prime}) \|_{\rho, s, b'}\\
	&= C \| u_{0} \|_{\mathcal{G}_{\rho, s}} + CT^{1-b+b'}
	\|  \partial_{x}\left(u^{p} v^{p+1}\right)\|_{\rho, s, b'}.
	\end{align*}
	We use Lemma \ref{1.3} to have 
	\begin{align*}
	\|  \partial_{x} u^{p} v^{p+1}\|_{\rho, s, b'}& \leq C \|  u \|^{p}_{\rho, s, b}    \|  v \|^{p+1}_{\rho, s, b}.
	\end{align*}
	Then 
	\begin{align*}
	\|  \varXi [u,v](t) \|_{\rho, s, b}
	&\leq C \| u_{0} \|_{\mathcal{G}_{\rho, s}} + CT^{1-b+b'}
	\|  u \|^{p}_{\rho, s, b}    \|  v\|^{p+1}_{\rho, s, b} \\
	&\leq C  \max \left( \| u_{0} \|_{\mathcal{G}_{\rho, s}}, \| v_{0} \|_{\mathcal{G}_{\rho, s}}\right)  + CT^{1-b+b'}
	\max \left( \| u \|_{\rho, s, b}, \| v \|_{\rho, s, b}\right)^{p}  \\& \quad \times \max \left( \| u \|_{\rho, s, b}, \| v \|_{\rho, s, b}\right)^{p+1}\\
	&\leq C  \max \left( \| u_{0} \|_{\mathcal{G}_{\rho, s}}, \| v_{0} \|_{\mathcal{G}_{\rho, s}}\right)  + CT^{1-b+b'}
	\max \left( \| u \|_{\rho, s, b}, \| v \|_{\rho, s, b}\right)^{2p+1}.
	\end{align*}
	The estimates for the second term  $ \Gamma $ are similar.
	\begin{align*}
	\|  \Gamma[u,v](t) \|_{\rho, s, b}   &\leq   C \| ( u_{0}, v_{0}) \|_{\mathcal{N}^{\rho, s}}
	+ CT^{1-b+b'}
	\left( \| (u, v) \|_{\mathcal{B}_{\rho, s, b}}\right)^{2p+1}.
	\end{align*}
	Then we have 
	\begin{align*}
	\|  \varXi [u,v](t), \Gamma[u,v](t) \|_{\mathcal{B}_{\rho, s, b}}   &\leq   C \| ( u_{0}, v_{0}) \|_{\mathcal{N}^{\rho, s}}
	+ CT^{1-b+b'}
	\left( \| (u, v) \|_{\mathcal{B}_{\rho, s, b}}\right)^{2p+1}.
	\end{align*}
	Then
	\begin{align*}
	\|  \varXi [u,v](t), \Gamma[u,v](t) \|_{\mathcal{B}_{\rho, s, b}}     &\leq  C \| ( u_{0}, v_{0}) \|_{\mathcal{N}^{\rho, s}}
	+ CT^{1-b+b'}
	\left( \| (u, v) \|_{\mathcal{B}_{\rho, s, b}}\right)^{2p+1} \\
	& \leq \frac{R}{2} +T^{\epsilon} C R^{2p+1}.
	\end{align*}
	We choose sufficiently small $T  $ such that
	\begin{equation*}
	T^{\epsilon}  \leq \frac{1}{4CR^{2p}}.
	\end{equation*}
	Hence
	\begin{equation*}
	\|  \varXi [u, v](t), \Gamma[u, v](t) \|_{\mathcal{B}_{\rho, s, b}} \leq R    \quad, \forall
	(u, v) \in \mathbb{B}(0, R).
	\end{equation*}
	Secondly we proof that  the map $  \varXi \times \Gamma: \mathbb{B}(0,R) \longrightarrow \mathbb{B}(0,R) $ is a contraction.\\	
	For this end, let  $  (u, v) \in \mathbb{B}(0,R)   $  and $  (u^{*}, v^{*}) \in \mathbb{B}(0,R) $ such that
	\begin{align*}
	\|  \varXi [u,v](t) - \varXi [u^{*},v^{*}](t) \|_{\rho, s, b} &= C\| \psi_{T}(t)\int_{0}^{t} W(t-t^{\prime})  \partial_{x}\left(u^{p} v^{p+1}-u^{*p} v^{*p+1}   \right)dt^{\prime}\|_{\rho, s, b} \\ \\
	&=C\| \psi_{T}(t)\int_{0}^{t} W(t-t^{\prime})  \partial_{x}\left[(u^{p}-u^{*p}) v^{p+1}+u^{*p} (v^{p+1}-v^{*p+1})   \right]dt^{\prime}\|_{\rho, s, b}.
	\end{align*}
	We use the Lemma  \ref{1.3} to have
		\begin{equation*}\label{p4}
		\begin{array}{l}
		\|  \partial_{x}\left(u^{p} - u^{*p}\right)v^{p+1} \|_{\rho, s, b}  \leq C   \| u^{p} - u^{*p}\|_{\rho, s, b'}   \| v^{p+1} \|_{\rho, s, b}
		\\ \\
		
		\|  \partial_{x}  u^{*p} \left( v^{p+1}- v^{*p+1}  \right)\|_{\rho, s, b'}	 \leq C  \|   u^{*p} \|_{\rho, s, b} \|   \left( v^{p+1}- v^{*p+1}  \right)\|_{\rho, s, b}.
		\end{array}
		\end{equation*}
	According to Lemma \ref{1.3}, we have  \\
	\begin{equation*}
	\begin{array}{l}
	\|  \left(u^{p} - u^{*p}\right) \|_{\rho, s, b}  \leq C   \|  \left(u - u^{*}\right)\|_{\rho, s, b}   R^{p-1}
		\\ \\
	\|   \left( v^{p+1}- v^{*p+1}  \right)\|_{\rho, s, b}\leq C  \|   \left( v- v^{*}  \right)\|_{\rho, s, b} R^{p}.
	\end{array}
	\end{equation*}
	Then
	\begin{equation*}
	\begin{array}{ll}
	\|  \partial_{x}\left(u^{p} - u^{*p}\right)v^{p+1} \|_{\rho, s, b'}
	& \leq C   \|  \left(u^{p} - u^{*p}\right)\|_{\rho, s, b}   \| v \|^{p+1}_{\rho, s, b} \\ \\
	
	&\leq C   \|  \left(u - u^{*}\right)\|_{\rho, s, b}   R^{p-1}  R^{p+1}
	\\ \\
	&\leq CR^{2p} \max \left(  \|(u - (u^{*})\|_{\rho, s, b},  \| v- v^{*}  \|_{\rho, s, b} \right)\\ \\
	&= CR^{2p} \|  (u- u^{*}),v- v^{*} \|_{\mathcal{B}_{\rho, s, b}},
	\end{array}
	\end{equation*}
	and
	\begin{equation*}\label{p7}
	\begin{array}{ll}
	\|  \partial_{x} u^{*p} \left( v^{p+1}- v^{*p+1}  \right)\|_{\rho, s, b}	
	\leq C R^{2p} \|  (u- u^{*},v- v^{*} \|_{\mathcal{B}_{\rho, s, b}},
	\end{array}
	\end{equation*}
	and
		\begin{equation*}
		\|  \varXi [u,v](t) - \varXi [u^{*},v^{*}](t) \|_{\rho, s, b} \leq  2CT^{1-b+b'}R^{2p} \|  u- u^{*},v- v^{*} \|_{\mathcal{B}_{\rho, s, b}},
		\end{equation*}
		\begin{equation*}
		\|  \Gamma [u,v](t) - \Gamma  [u^{*},v^{*}](t) \|_{\rho, s, b} \leq  2CT^{1-b+b'}R^{2p} \|  u- u^{*},v- v^{*} \|_{\mathcal{B}_{\rho, s, b}}.
		\end{equation*}
	By the same way we prove that $\Gamma [u,v](t) $
	is contraction, so we have
	\begin{align*}
	\|  \varXi [u,v](t) - \varXi [u^{*},v^{*}], \Gamma [u,v](t) - \Gamma [u^{*},v^{*}]\|_{\mathcal{B}_{\rho, s, b}} \\	
	\leq 2 C  ~T^{\epsilon} R^{2p} \|  u- u^{*}, v- v^{*})\|_{\mathcal{B}_{\rho, s, b}}.
	\end{align*}
	Since $ T^{\epsilon}  \leq \frac{1}{4CR^{2p}}$, we have
	\begin{align*}
	\|  \varXi [u,v] - \varXi [u^{*},v^{*}](t), \Gamma [u,v] - \Gamma [u^{*},v^{*}](t)\|_{\mathcal{B}_{\rho, s, b}} \\
	\leq \frac{1}{2} \|  (u- (u^{*})
	,v- v^{*} \|_{\mathcal{B}_{\rho, s, b}}.
	\end{align*}
	Since   the map $  \varXi \times \Gamma: \mathbb{B}(0,R) \longrightarrow \mathbb{B}(0,R) $ is a contraction, it follows that   has a unique fixed point $(u,v)$  in $B(0, R)$.
\end{proof}
The rest of the proof follows a standard argument.
\section{Large time estimates on the radius of analyticity.}
\begin{lemma}\label{lem23}
	Let $   s > \frac{3}{2}  $,$  \rho > 0 $, $ T\geq 1 $ and  $  b \in [-1, 1  ]$. We suppose that $ ( u,v ) $ is solution of $ (\ref{p01} )$ on the time interval $   [0, 2T   ]$. Then there exists a constants $ C $ such that 
		\begin{equation}\label{eq250}
		\| (\psi_{T}(t) u(.,t), \psi_{T}(t) v(.,t)) \|_{\mathcal{B}_{s,b}}\leq C T^{\frac{1}{2}} \left( 1+ \lambda_{T} (u,v)    \right)^{2p+1},
		\end{equation}
		and
		\begin{equation}\label{eq201}
		\| \psi_{T}(t) u(.,t), \psi_{T}(t) v(.,t) \|_{\mathcal{B}_{\rho,s,b}}\leq CT^{\frac{1}{2}} \left( 1+\kappa_{T} (u,v)    \right)^{{2p+1}},
		\end{equation}
		with
		\begin{equation*}
		\lambda_{T} (u,v)= \sup_{\substack{ t \in [0, 2T]
		}}  \left(    \| u, v   \|_{\mathcal{N}^{s+1}}  \right)	 \quad \textit{and} \quad  \kappa_{T} (u,v)= \sup_{\substack{ t \in [0, 2T]
		}}  \left(    \| u, v   \|_{\mathcal{N}^{\rho, s+1}}  \right),
		\end{equation*}	
		where $\mathcal{N}^{s}=H^{s}\times H^{s}$ and $\mathcal{B}_{b,s}=X_{b,s}\times X_{b,s} $.
\end{lemma}
\begin{proof} We have
	\begin{align*}
	\| \psi_{T}(t) u(x,t) \|^{2}_{s,b} & 	=\int_{-\infty}^{+\infty}  \left(1+| \zeta |  \right)^{2s}    \int_{-\infty}^{+\infty} \bigg| \Lambda^{b} \left(  e^{-it\zeta^{3}} \psi_{T}(t) \widehat{u}^{x}(\zeta,t)   \right)\bigg|^{2} dt d \zeta.
	\end{align*}
	By using  the inquality
	$$ |\Lambda^{b}v(x,t)| \leq c   |v(x,t) |+ |\partial_{t}v(x,t) |,   $$
	we get
	\begin{align*}
	\| \psi_{T}(t) u(.,t) \|^{2}_{s,b} &\leq  c \int_{-\infty}^{+\infty}  \left(1+| \zeta |  \right)^{2s}    \int_{-\infty}^{+\infty}\bigg | \left(  e^{-it\zeta ^{3}} | \psi_{T}(t)  \widehat{u}^{x}  (\zeta,t)   \right)\bigg|^{2} dt d \zeta
	\\& \quad
	+ c \int_{-\infty}^{+\infty}  \left(1+| \zeta |  \right)^{2s}    \int_{-\infty}^{+\infty} \bigg|\partial_{t} \left(  e^{-it(\zeta )^{3}} | \psi_{T}(t)  \widehat{u}^{x}  (\zeta,t)   \right)\bigg|^{2} dt d \zeta.
	\end{align*}\\
	We have 	
	\begin{align*}
	\partial_{t} \left(  e^{-it\zeta ^{3}}  \psi_{T}(t)  \widehat{u}^{x} (\zeta,t)   \right)= &\frac{1}{T}\psi^{\prime}_{ T}(t) e^{-it\zeta ^{3}}\widehat{u}^{x} (\zeta,t)  + \psi_{ T}(t)  (  -i\zeta ^{3})e^{-it(\zeta )^{3}}\widehat{u}^{x} (\zeta,t)  \\ & + \psi_{ T}(t) e^{-it\zeta ^{3}}\widehat{u}^{x}_{t} (\zeta,t),
	\end{align*}
	and
	\begin{equation*}
	  u_{t}= -\partial_{x}^{3} u-\partial_{x}\left(u^{p} v^{p+1}\right).
	\end{equation*}	
	Then
	\begin{align*}
	\widehat{u}^{x}_{t} (\zeta,t) & = - \widehat{\partial^{3}_{x}u }^{x} (\zeta,t) - \widehat{\partial_{x}\left(u^{p} v^{p+1}\right)}^{x} (\zeta,t)
	\\&
	= i\zeta^{3}\widehat{u}^{x}(\zeta,t) -i \zeta  \widehat{\left(u^{p} v^{p+1}\right)  }^{x}(\zeta,t).
	\end{align*}
	So
	\begin{equation*}
	\partial_{t} \left(  e^{-it(\zeta )^{3}}  \psi_{T}(t)  \widehat{  u }^{x}(\zeta,t)   \right)= \frac{1}{T}\psi^{\prime}_{ T}(t) e^{-it\zeta ^{3}} \widehat{u}^{x} (\zeta,t)  +  \psi_{ T}(t) e^{-it\zeta ^{3}}  i \zeta   \widehat{  ( u^{p} v^{p+1}  )  }^{x} (\zeta,t),
	\end{equation*} 
	and
	\begin{align*}
	\| \psi_{T}(t) u(.,t) \|^{2}_{s,b}& = \int_{-\infty}^{+\infty}  \left(1+| \zeta |  \right)^{2s}    \int_{-\infty}^{+\infty}\bigg | \Lambda^{b} \left(  e^{-it\zeta^{3}}  \psi_{T}(t)  \widehat{u}^{x} (\zeta,t)   \right)\bigg |^{2} dt d \zeta
	\\&
	\leq  c \int_{-\infty}^{+\infty}  \left(1+| \zeta |  \right)^{2s}    \int_{-\infty}^{+\infty} \bigg | \left(  e^{-it\zeta ^{3}}  \psi_{T}(t)  \widehat{u}^{x} (\zeta,t)   \right)\bigg |^{2} dt d \zeta
	\\& \quad
	+ c \int_{-\infty}^{+\infty}  \left(1+| \zeta |  \right)^{2s}    \int_{-\infty}^{+\infty} \bigg | \left(  e^{-it\zeta ^{3}}   \frac{1}{T}\psi^{\prime}_{T}(t)  \widehat{u}^{x} (\zeta,t)   \right)\bigg |^{2} dt d \zeta
	\\& \quad
	+ c \int_{-\infty}^{+\infty}  \left(1+| \zeta |  \right)^{2s}    \int_{-\infty}^{+\infty} \bigg | \left(  e^{-it\zeta^{3}}  \psi_{T}(t)(i \zeta ) \widehat{u^{p} v^{p+1}}^{x} (\zeta,t)   \right)\bigg |^{2} dt d \zeta
	\\
	&\leq 2 c \int_{-\infty}^{+\infty}  \left(1+| \zeta |  \right)^{2s}    \int_{0}^{+2T} | \left(  \widehat{u}^{x}  (\zeta,t)   \right)|^{2} dt d \zeta
	+ c \int_{-\infty}^{+\infty}  \left(1+| \zeta |  \right)^{2s}   \\& \quad \times  \int_{0}^{2T} \bigg | \left(  |\zeta | \widehat{ u^{p}  v^{p+1} }^{x}   (\zeta,t)   \right)\bigg |^{2} dt d \zeta,	
	\end{align*}
	and
	\begin{align*}
	\| \psi_{T}(t) u(.,t) \|^{2}_{s,b}&\leq  4 c T \sup_{\substack{t \in [0, 2T]}}\| u(., t)\|^{2}_{H^{s}} + 2cT \sup_{\substack{ t \in [0, 2T]}}( \| u^{p}v^{p+1}  \|^{2}_{H^{s+1}})
	\\ &
	\leq 4 c T \sup_{\substack{t \in [0, 2T]}}\| u(., t)\|^{2}_{H^{s}} + 2cT \sup_{\substack{ t \in [0, 2T]}} (\| u^{p} \|_{H^{s+1}}  \| v^{p+1}  \|_{H^{s+1}} )^{2}
		\\ &
	\leq  4 c T \sup_{\substack{t \in [0, 2T]}}(\| (u, v)\|_{\mathcal{N}^{s}} )^{2} + 2cT \sup_{\substack{ t \in [0, 2T]}} ((\| (u, v )\|_{\mathcal{N}^{s+1}} )^{2p+1} )^{2},
	\end{align*}
	and
	\begin{equation*}
	\| \psi_{T}(t) u(.,t) \|_{s,b} \leq cT^{\frac{1}{2}} \left( 1+\lambda_{T} (u,v)    \right)^{2p+1},
	\end{equation*}
	where
	$$    \lambda_{T} (u,v) =   \sup_{\substack{ t \in [0, 2T]}} (\| (u, v )\|_{\mathcal{N}^{s+1}} ).   $$
	Similarity,
	$$	\| \psi_{T}(t) v(.,t) \|_{s,b} \leq cT^{\frac{1}{2}} \left( 1+\lambda_{T} (u,v)    \right)^{2p+1}, $$
	and
	\begin{align*}
 \| (\psi_{T}(t) u(.,t), \psi_{T}(t) v(.,t)) \|_{s,b} & \leq 2c~T^{\frac{1}{2}} \left( 1+\lambda_{T} (u;v)    \right)^{2p+1}\\
	&\leq CT^{\frac{1}{2}} \left( 1+\lambda_{T} (u;v)    \right)^{2p+1}.
	\end{align*}
This complets the proof.
\end{proof}
To prove the Theorem \ref{th03}, we need  to define a sequence of approximations to $  (\ref{p01})$ as follows 
\begin{equation}
\left\{\begin{array}{l}\label{eq10}
u^{n}_{t}+\partial_{x}^{3} u^{n}=-\partial_{x}\left((  \rho_{n}\ast \psi_{T} u^{n})^{p}) (  \rho_{n}\ast \psi_{T} v^{n})^{p+1})\right), \\
v^{n}_{t}+\partial_{x}^{3} v^{n}= -\partial_{x}\left(( \rho_{n}\ast \psi_{T} u^{n})^{p+1}) (  \rho_{n}\ast \psi_{T} v^{n})^{p})\right), \quad x, t \in \mathbb{R}, p \in \mathbb{Z}^{+} \\
u^{n}(x, 0)=u_{0}(x), \quad v^{n}(x, 0)=v_{0}(x),
\end{array}\right.
\end{equation}
where $ T > 0 $, $ n \in \mathbb{N}$
and   $ \rho_{n} $ is defined as
\begin{equation*}
\widehat{\rho}_{n}(\zeta) =
\left\{\begin{array}{l}
0,  \quad | \zeta | \geq 2n
\\
1, \quad | \zeta | \leq n,
\end{array}\right.
\end{equation*}
where $\widehat{\rho}_{n}  $ is smooth and monotone on $  ( n, 2n)$.
\begin{lemma}\label{lem043}
	Let $  s \geq 0  $ and $  (u_{0},v_{0} ) \in \mathcal{N}^{s} $ and we assume that $( u, v)  $  is solution  of $ (\ref{p01} )$ with  $  (u_{0},v_{0} )$.  Then for $ n \in \mathbb{N}$, we have
	\begin{itemize}
		\item  $  (u^{n}, v^{n} )$ is in $ C([0,2T], H^{s}   ) \times C([0,2T], H^{s}   )   $. The sequence $  \{ (u^{n },v^{n} )\} $ converge  to $  (u,v)  $ \\
		in
		$ C([0,T], H^{s}   ) \times C([0,T], H^{s}   )  $.
		\item  The estimate in  Lemma \ref{lem23} holds for  $   (u^{n},v^{n} )  $ uniformly in $ n $.
		\item If $  (u_{0}, v_{0} ) \in \mathcal{N}^{\rho, s} $ for $ \rho > 0  $, then
		the result is also given for $ C([0,T], \mathcal{G}_{\rho, s}   ) \times  C([0,T], \mathcal{G}_{\rho, s}   ).  $
	\end{itemize}
\end{lemma}
\begin{lemma}\label{lem4} (\cite{130})
	Let $ (u, v)   $ be solution of $(\ref{p01}) $  with the initial data $(u_{0},v_{0} )\in \mathcal{N}^{\rho_{0}, s+1}$ for $ \rho_{0} > 0  $ and $ s >\frac{3}{2}    $ and $\eta > 0   $, then
	\begin{align*}
	\sup_{\substack{ t \in [0, 2\eta]}} \Vert (u(., t), v(., t))           \Vert_{\mathcal{N}^{\rho(t), s+1}} \leq \Vert (u_{0},v_{0} )  \Vert_{\mathcal{N}^{\rho_{0}, s+1}}  + C \eta^{\frac{1}{2}}  \sup_{\substack{ t \in [0, 2\eta]}} \Vert (u(., t), v(., t))           \Vert^{(2p+2)/2}_{ \mathcal{N}^{s+1}},
	\end{align*}
	with $   \rho (t)= \rho_{0} e^{-\gamma (t)}        $
	and $\gamma (t)$ is defined as
	$$     \gamma (t)= \int_{0}^{t} \left( k_{1}   + k_{2} \int_{0}^{t^{\prime}}  \Vert \left(  u(., t^{\prime \prime}    ), v(., t^{\prime \prime})   \right)         \Vert^{2p+2}_{ \mathcal{N}^{s+1}} dt^{\prime \prime} \right)^{2p}  dt^{\prime},       $$
	where
	$$    k_{1}=       \Vert (u_{0},v_{0} )  \Vert^{2}_{\mathcal{N}^{\rho_{0}, s+1}} ,  $$ and $ k_{2}  $ is a constant.
\end{lemma}

\begin{proposition}\label{pro02}
	Let $ \rho_{0} > 0 $, $  p \geqslant 1   $,$  T \geqslant 1$    and $  s > 3b  $, we assume  that $( u, v)     $
	is solution of $ (\ref{p01})   $ in $   C \left( [0, 2T], H^{s+1}  \right) \times  C \left( [0, 2T], H^{s+1}  \right)                              $
	with $ ( u_{0}, v_{0} )  \in  \mathcal{G}_{\rho_{0}, s+1} \times \mathcal{G}_{\rho_{0}, s+1}    $, then  there exist  $   \rho_{1} < \rho_{0}   $ such that
	
	$$      \lbrace \Psi_{T} u^{n}, \Psi_{T} v^{n} \rbrace \quad  \text{bounded in} \quad  \mathcal{B}_{\rho (t), s,b},  $$
	with
	$$ \rho (t) \leq \min \lbrace   \rho_{1},K T ^{-2p^{2} -6p-1}  \rbrace . $$
	
\end{proposition}

\begin{proof}
	We have
	\begin{align}
	\psi_{T}(t) u^{n} = \psi_{T}(t) W(t) u_{0}- \psi_{T}(t) \int_{0}^{t} \partial_{x}\left(( \rho_{n}\ast \psi_{T} u^{n})^{p+1}) (  \rho_{n}\ast \psi_{T} v^{n})^{p})\right) dx,
	\end{align}
	where $    t \in (0, \infty )$.  This will show that $\Psi_{T} u^{n}    \in  X_{\rho,s,b} $ for all $ n  \in \mathbb{N}$. \\
	We have
	\begin{align*}
	\|  \psi_{T}(t) u^{n}   \|_{\rho, s,b} &\leq  \|  \psi_{T}(t) W(t) u_{0}   \|_{\rho, s,b} +  \| \psi_{T}(t) \int_{0}^{t} \partial_{x}\left( \rho_{n}\ast \psi_{T} (u^{n})^{p+1}) (  \rho_{n}\ast \psi_{T} (v^{n})^{p} \right)   \|_{\rho, s,b}\\&
	\leq c T^{\frac{1}{2}} \|  u_{0} \|_{\mathcal{G}_{\rho, s} } +c T \|  \partial_{x}\left( ( \rho_{n}\ast \psi_{T}  u^{n})^{p+1} (\rho_{n}\ast \psi_{T} (v^{n}))^{p}\right)    \|_{\rho, s,b'}
	\\ &\leq c T^{\frac{1}{2}} \|  u_{0} \|_{\mathcal{G}_{\rho, s} } +c T  \left(     \|   \psi_{T} u^{n}\|_{s,b}^{p+1}   \|   \psi_{T} v^{n}\|_{s,b}^{p} +
	\rho^{\frac{1}{2}} \|   \psi_{T} u^{n}\|_{ \rho,s,b}^{p+1}   \|   \psi_{T} v^{n}\|_{\rho,s,b}^{p} \right).
	\end{align*}
	For $ 0 < \rho    < \rho_{0}   $ and $    b' = b-1+ \epsilon ' $
	where $ \epsilon '> 0  $ and we use the Lemma \ref{lem23} to obtain
	$$   \|  \psi_{T}(t) u^{n}   \|_{ s,b}  \leq c T^{\frac{1}{2}} (1+\alpha_{T}(u^{n}, v^{n} ) )^{2p+1}
	\leq   2 c T^{\frac{1}{2}} (1+\alpha_{T}(u,v ) )^{2p+1},  $$
	and
	$$   \|  \psi_{T}(t) v^{n}   \|_{ s,b}  \leq c T^{\frac{1}{2}} (1+\alpha_{T}(u^{n}, v^{n}  ) )^{2p+1}
	\leq   2 c T^{\frac{1}{2}} (1+\alpha_{T}(u,v ) )^{2p+1}.  $$
	Then
	\begin{align*}
	\|  \psi_{T}(t) u^{n}   \|_{\rho(t), s,b} & \leq   c T^{\frac{1}{2}} \|  u_{0} \|_{\mathcal{G}_{\rho (t), s} }+  c T ^{\frac{2p+3}{2}} (1+\alpha_{T}(u,v ) )^{(2p+1)^{2}}   +
	~ c ~T^{\frac{1}{2}}  (\rho(t))^{\frac{1}{2}}  \|  \psi_{T} u^{n}\|_{ \rho (t),s,b}^{p+1}   \|   \psi_{T} v^{n}\|_{\rho(t),s,b}^{p}
	\\& \leq   c T^{\frac{1}{2}} \| ( u_{0}, v_{0} ) \|_{\mathcal{N}^{\rho(t), s} }+  c T ^{\frac{2p+3}{2}} (1+\alpha_{T}(u,v ) )^{(2p+1)^{2}}   +  c T^{\frac{1}{2}}  \rho (t)^{\frac{1}{2}}   \| (\Psi_{T} u^{n},  \Psi_{T} v^{n})\|_{\mathcal{B}_{ \rho(t),s,b}}^{2p+1}.
	\end{align*}
	holds for $   T \geq 1 $.\\
	{\bf In  the case  $   T = 1 $}, and by using Lemma \ref{lem23}  and Lemma \ref{lem4}, we have
	\begin{equation*}
	\| (\psi_{T}(t) u(.,t), \psi_{T}(t) v(.,t)) \|_{\mathcal{B}_{\rho(1),s,b}}\leq ~c~ T^{\frac{1}{2}} \left( 1+\kappa_{T} (u,v)    \right)^{{2p+1}},
	\end{equation*}
	where 
	\begin{equation*}
	\kappa_{T} (u,v)= \sup_{\substack{ t \in [0, 2]
	}}  \left(    \| (u, v  ) \|_{\mathcal{N}^{\rho, s+1}}  \right)^{2p+1},
	\end{equation*}
	\begin{align*}
	\|  \psi_{1}(t) u^{n}   \|_{\rho (1), s,b} & \leq   c \left( 1+    \sup_{\substack{ t \in [0, 2]
	}}  \left(    \| (u^{n}, v^{n} )  \|_{\mathcal{N}^{\rho(1), s+1}}  \right)     \right)^{2p+1} \\&
	\leq   2c  \left( 1+    \sup_{\substack{ t \in [0, 2]
	}}  \left(    \| (u, v )  \|_{\mathcal{N}^{\rho(1), s+1}}  \right)  \right)^{2p+1} \\&
	\leq 2cc_1 \left(  1+ \| ( u_{0},  v_{0} )\|^{2p+1}_{\mathcal{N}^{\rho(1), s+1} } + \sup_{\substack{t \in [0, 2]}}
	\left(    \| (u, v )  \|_{\mathcal{N}^{ s+1}}     \right)^{((2p+2)(2p+1))/2}  \right).
	\end{align*}
	We assume that
	$$   M^{\ast} =   2cc_1 \left( 1+ \| ( u_{0}, v_{0} )\|^{2p+1}_{\mathcal{N}^{\rho(1), s+1} } + \sup_{\substack{t \in [0, 2]}}
	\left(    \| (u, v)   \|_{\mathcal{N}^{ s+1}}     \right)^{((2p+2)(2p+1))/2} \right).        $$
	Then
	\begin{eqnarray}
	 \|  \psi_{T}(t) u^{n}     \|_{\rho(t), s,b}
	&\leq&  M^{\ast} +  c T^{\frac{1}{2}} \| ( u_{0}, v_{0} ) \|_{\mathcal{N}^{\rho_{0}, s} }+  c T ^{\frac{p+3}{2}} (1+\alpha_{T}(u,v ) )^{(2p+1)^{2}}   \nonumber\\
	&+&  c T^{\frac{1}{2}}  \rho(t)^{\frac{1}{2}}   \| ( \Psi_{T} u^{n}, \Psi_{T} v^{n})\|_{\mathcal{B}_{ \rho(t),s,b}}^{2p+1},\nonumber
	\end{eqnarray}
	and
	\begin{eqnarray}
	\| ( \psi_{T}(t) u^{n},\psi_{T}(t) v^{n} ) \|_{\mathcal{B}_{\rho(t), s,b}}
	&\leq&  M^{\ast} +  c T^{\frac{1}{2}} \| ( u_{0}, v_{0} ) \|_{\mathcal{N}^{\rho_{0}, s} }+  c T ^{\frac{p+3}{2}} (1+\alpha_{T}(u,v ) )^{(2p+1)^{2}}   \nonumber\\
	&+&  c T^{\frac{1}{2}}  \rho(t)^{\frac{1}{2}}   \| (\Psi_{T} u^{n},\Psi_{T}  v^{n})\|_{\mathcal{B}_{ \rho(t),s,b}}^{2p+1}.\nonumber
	\end{eqnarray}
	{For $ T \geq 1  $}, $ \rho(t) \leq \rho_{1} \leq \rho_{0}  $, and for large enough  $n$, we define the new variables 
	\begin{align*}
	& y=y (T)= \|  \psi_{T}(t) u^{n}, \psi_{T}(t) v^{n}  \|_{\mathcal{B}_{\rho(t), s,b}} \\&
	x = x(T)= M^{\ast} + c T^{\frac{1}{2}} \| ( u_{0}, v_{0} ) \|_{\mathcal{N}^{\rho_{0}, s} }+  c T ^{\frac{p+3}{2}} (1+\alpha_{T}(u,v ) )^{(2p+1)^{2}}  \\ &
	d= d(T)= c T^{1/2}.
	\end{align*}
	Then
	$$  y \leq x+ d \rho (T)^{\frac{1}{2}}  y^{2p+1}.     $$
	If define
	$$\rho (T) = \frac{a^{2}}{d^{2} x^{4p} 2^{4p}}.         $$
	Then
	$$  y \leq x+ d \rho (T)^{\frac{1}{2}}  y^{2p+1} \leq  x+ d  (\frac{a^{2}}{d^{2} x^{4p} 2^{4p}})^{\frac{1}{2}} y^{2p+1} \leq  x  +(\frac{a}{ (2 x)^{2p}  })   y^{2p+1} $$
	$$  \Longrightarrow y  \leq  x  +a \left( \frac{y}{ 2 x}\right) ^{2p}    y  \Longrightarrow \frac{y}{2x} \leq  \frac{1}{2}  +a (\frac{y}{ 2 x})^{2p+1}.           $$
	We define $ h(t) = \frac{y(t)}{ 2 x(t)}$. Then
	$$  h (1- a h^{2p} ) \leq   \frac{1}{2}.    $$
	We can choose small $ a $ for all $ p $, then  there is $ M'  $ and $  m'  $
	such that
	$$    \frac{1}{2}  < m' <  1 < M',   $$
	and
	$$  h \leq m'    \quad      or  \quad   h \geq M'.     $$
	As $  \|  \psi_{T}(t) u^{n},  \psi_{T}(t) v^{n}   \|_{\mathcal{B}_{\rho(t), s,b}}   $ is a continuous function of $ T \geq 1$, then
	$$   h(t) \geq m' < 1  \Longrightarrow  y(t) \leq 2x(t),     $$
	which means that
	$$      \|  \psi_{T}(t) u^{n}, \psi_{T}(t) v^{n}  \|_{\mathcal{B}_{\rho(t), s,b}}  \leq 2 x.         $$
	Then
	$$      \lbrace \Psi_{T} u^{n} \rbrace   ~~ ~~  and  ~~~   \lbrace \Psi_{T} v^{n} \rbrace  ~~~ \textit{bounded in }  X_{\rho (t), s,b}.  $$
On the other hand, we have
	\begin{equation}
	\left\{\begin{array}{l}
	\rho (t) < \rho_{1}\\
	\rho (t)=   \frac{a^{2}}{d^{2} x^{4p} 2^{4p}}.
	\end{array}\right.
	\end{equation}
	Since
	\begin{align*}
	x^{4p} = (x(T))^{4p}&=\left(    M^{\ast} + c T^{\frac{1}{2}} \| ( u_{0}, v_{0} ) \|_{\mathcal{N}^{\rho_{0}, s} }+  c T ^{\frac{p+3}{2}} (1+\alpha_{T}(u,v ) \right)^{4p}\\& \geq
	\left(     c T^{\frac{1}{2}} \| ( u_{0}, v_{0} ) \|_{\mathcal{N}^{\rho_{0}, s} }+  c T ^{\frac{p+3}{2}} (1+\alpha_{T}(u,v ) \right)^{4p}
	\\& \geq
	T^{\frac{4p}{2}}  \left(     c  \| ( u_{0}, v_{0} ) \|_{\mathcal{N}^{\rho_{0}, s} }+  c T ^{\frac{p+2}{2}} (1+\alpha_{T}(u,v ) \right)^{4p}\\& \geq
	T ^{2p} \left(       c T ^{\frac{p+2}{2}} (1+\alpha_{T}(u,v ) \right)^{4p} \\ &
	= T ^{2p^{2} +6p } (1+\alpha_{T}(u,v )^{4p}.
	\end{align*}
	Then
	\begin{align}
	x^{-4p} \leqslant (T )^{-2p^{2} -6p } (1+\alpha_{T}(u,v)^{-4p},
	\end{align}
	and
	\begin{align*}
	\rho (t) & =   \frac{a^{2}}{d^{2} x^{4p} 2^{4p}}=  \dfrac{a^{2}}{  c^{2} T x^{4p} 2^{4p}  }= \dfrac{a^{2}T^{-1} }{  c^{2}  x^{4p} 2^{4p}  } \leqslant  \dfrac{a^{2}T^{-1} T ^{-2p^{2} -6p }  }{  c^{2} ((1+\alpha_{T}(u,v ))^{4p}   2^{4p}  } \\ \\ \rho (t) &
	\leqslant  \dfrac{a^{2}  }{  c^{2} ((1+\alpha_{T}(u,v ))^{4p}   2^{4p}  }  T ^{-2p^{2} -6p-1} =  K T ^{-2p^{2} -6p-1},
	\end{align*}
	where
	$$ K =   \dfrac{a^{2}  }{  c^{2} ((1+\alpha_{T}(u,v ))^{4p}   2^{4p}  }.  $$
	and
 $$ \rho (t) =  \min \left\lbrace   \rho_{1},  K T ^{-2p^{2} -6p-1}  \right\rbrace. $$

\end{proof}
We are now in potion to prove  Theorem \ref{th03}.
\begin{proof}{Of Theorem \ref{th03}.}
We have $ ( u_{0}, v_{0} )  \in  \mathcal{N}^{\rho_{0}, s+1} $,
then by Theorem  \ref{the1.2}, we obtain 
$$ ( u, v )  \in  C([0, T^{*}], \mathcal{G}_{\rho_{0}, s+1} ) \times C([0, T^{*}],\mathcal{G}_{\rho_{0}, s+1} ).   $$
We prove that
$$    (u,v )\in    C \left( [0, T], \mathcal{G}_{\frac{\rho (t)}{2}, s+1} \right) \times  C \left( [0, T],\mathcal{G}_{\frac{\rho (t)}{2}, s+1}  \right).                                $$
If $  T^{*} = \infty,$ it is done.\\
If $   T^{*} < \infty           $, it remains to prove that
$$    (u,v )\in    C \left( [0, T], \mathcal{G}_{\frac{\rho (t)}{2}, s+1} \right) \times  C \left( [0, T], \mathcal{G}_{\frac{\rho (t)}{2}, s+1}  \right),                    \quad \forall  ~~ T \geqslant T^{*}.             $$
From the Proposition \ref{pro02}, we obtain that  the sequence $ \lbrace  ( u^{n}, v^{n}  ) \rbrace   $  is solution of $(\ref{eq10})   $ where $ (u_{0}, v_{0}      )   $ is bounded in    $ \mathcal{G}_{\rho (t), s}  $ uniformly on $ [0, T  ]   $.\\
	By  using  Lemma \ref{lem2},  with $ (u^{n}, v^{n} )$  satisfies $(\ref{eq10} )$ then, we obtain
	\begin{align*}
	( \partial_{t} u^{n},\partial_{t} v^{n} )
	\quad
	( \partial_{x} u^{n}, \partial_{x} v^{n} )
	\quad
	( \partial_{x}^{3} u^{n}, \partial_{x}^{3} v^{n} )
|\quad  \textit{are uniformly  bounded on the  strip} \quad  G_{\frac{\rho(t)}{2}, s}.
	\end{align*}
	Then
	\begin{align*}
	( \partial_{t} u^{n},\partial_{t} v^{n} )
	\quad
	( \partial_{x} u^{n}, \partial_{x} v^{n} )
	\quad
	( \partial_{x}^{3} u^{n}, \partial_{x}^{3} v^{n} )
|\quad  \textit{   are equicontinuous families on strip } \quad  G_{\frac{\rho(t)}{2}, s}.
	\end{align*}	
	 Then, we can extract a subsequence   (without changing symbol of $ \lbrace  ( u^{n}, v^{n}  ) \rbrace   $ ) converging uniformly on compact subsets of  $ (0, T  ) \times G_{\frac{\rho(t)}{2}, s }    $  to smooth function $ ( \tilde{u},\tilde{v}  )$ and
	\begin{equation*} 
	( \partial_{t} u^{n},\partial_{t} v^{n} )
	\quad
	( \partial_{x} u^{n}, \partial_{x} v^{n} )
	\quad
	( \partial_{x}^{3} u^{n}, \partial_{x}^{3} v^{n} ) \quad
	\textit{ is  converging uniformly on compact subsets of}   \quad  (0, T  ) \times G_{\frac{\rho(t)}{2}, s }.
	\end{equation*}
Next we passe to the  limit in $   (\ref{eq10} ) $, we obtain that $ (\tilde{u}, \tilde{v})   $ is  a smooth extension of  $ (u,v)$.\\	
	Since,  $   ( u^{n}, v^{n}  )     $  is analytic   $  G_{\frac{\rho(t)}{2}, s }    $ to the $  (\tilde{u}, \tilde{v})$, so $(\tilde{u}, \tilde{v})    $ is analytic in $   G_{\frac{\rho(t)}{2}, s }$, on the other hand, since    $ \lbrace  ( u^{n}, v^{n}  ) \rbrace   $   is bounded in $  G_{\frac{\rho(t)}{2}, s }   $ uniformly on  $ [0, T   ],$ then
	$$  \tilde{u}  \equiv u \in L^{\infty} ( (0, T),\mathcal{G}_{\frac{\rho(t)}{2}} )   ~~~ \quad, ~~~~~  \tilde{v}  \equiv   v      \in L^{\infty} ( (0, T),\mathcal{G}_{\frac{\rho(t)}{2}} ),  $$
	then	
	$$ u \in C( (0, T),\mathcal{G}_{\frac{\rho(t)}{2}} )   ~~~ \quad, ~~~~~     v      \in C ( (0, T),\mathcal{G}_{\frac{\rho(t)}{2}} ). $$
\end{proof}

{\small 
\address{\textbf{Amel Atmani} \newline
	\textit{Department of Mathematics  Bordj Bou Arreridj University,  Algeria. Email address: amel.atmani@univ-bba.dz}}

\address{\textbf{Aissa Boukarou} \newline
\textit{Department of Mathematics, University of Ghardaia, Algeria. Email address: boukarouaissa@gmail.com}}

\address{\textbf{Djamila Benterki} \newline
	\textit{Department of Mathematics  Bordj Bou Arreridj University,  Algeria. Email address:benterkidj@yahoo.fr}}

\address{\textbf{Khaled zennir} \newline
	\textit{Department of Mathematics, College of Sciences and Arts, Qassim University, Ar-Rass, Saudi Arabia Email address:k.zennir@qu.edu.sa \\
Laboratoire de Math\'ematiques Appliqu\'ees et de Mod\'elisation, Universit\'e 8 Mai 1945 Guelma. B.P. 401 Guelma 24000 Alg\'erie.}}


\begin{thebibliography}{99}
	
	\bibitem{17}
	M. Ablowitz, D. Kaup, A. Newell and H. Segur, Nonlinear evolution equations of physical
	significance, Phys. Rev. Lett., 1973, 31(2), 125–127.
	
	\bibitem{7}
	E. Alarcon, J. Angulo and J. F. Montenegro, Stability and instability of solitary waves for
	a nonlinear dispersive system, Nonl. Anal., 1999, 36, 1015-1035.
	
	\bibitem{9}
	J. Angulo, J. Bona, F. Linares and M. Scialom, Scaling, stability and singularities for nonlinear dispersive wave equations: the critical case, Nonlinearity, 2002, 15, 759-786.
	
	\bibitem{4}
	J. L. Bona, Z. Gruji\'c and H. Kalisch, Algebraic lower bounds for the uniform radius of spatial analyticity for the generalized KdV equation, Ann. Inst. Henri Poincare, Anal. Non Lineaire, 2005, 22, 783-797.
	
	\bibitem{130}
	J. L. Bona, Z. Grujic, Spatial analyticity for nonlinear waves, Math. Models Methods Appl. Sci., 2003, 13, 1-15.
	
	\bibitem{22}
	L. J. Bona, Z. Gruji\'c and H. Kalisch. A KdV-type Boussinesq system: From the energy level to analytic spaces. Discrete \& Continuous Dyn. Syst., 2010, 26 (4): 1121-1139.
	
\bibitem{BoukarouA2}
A. Boukarou, K. Guerbati, Kh. Zennir, S. Alodhaibi and S. Alkhalaf,  {\em Well-Posedness and Time Regularity for a System of Modified Korteweg-de Vries-Type Equations in Analytic Gevrey Spaces}, Mathematics 2020, 8, 809.

\bibitem{Boukarou2}
A. Boukarou, Kh. Zennir, K. Guerbati, S. G. Georgiev, {\em Well-posedness and regularity of the fifth order Kadomtsev-Petviashvili I equation in the analytic Bourgain spaces}, Ann. Univ. Ferrara Sez. VII Sci. Mat., 2020, 66, 255-272.

\bibitem{Boukarou3} 
A. Boukarou, Kh. Zennir, K. Guerbati, S. G. Georgiev, {\em Well-posedness of the Cauchy problem of Ostrovsky equation in analytic Gevrey spaces and time regularity}, Rend. Circ. Mat. Palermo 2,2021, 70,  349-364.

\bibitem{Boukarou4} 
A. Boukarou, K. Guerbati, Kh. Zennir,   {\em On the radius of spatial analyticity for the higher order nonlinear dispersive equation}, Mathematica Bohemica, 2021, 1-14. 

\bibitem{Boukarou6} 
A. Boukarou, K. Guerbati, Kh. Zennir,   {Local well-posedness and time regularity for a fifth-order shallow water equations in analytic Gevrey–Bourgain spaces. }, Monatsh Math.,2020, 193, 763–782.

\bibitem{Boukarouarxiv}
A. Boukarou, D. Oliveira da Silva, K. Guerbati and Kh. Zennir,  {\em Global well-posedness for the fifth-order Kadomtsev-Petviashvili II equation in anisotropic Gevrey Spaces}, Dyn. Part. Diff. Equ., 2021, 18(2), 101-112. 

\bibitem{L1}
J. Bourgain, Fourier transform restriction phenomena for certain lattice subsets and applications to nonlinear evolution equations, Geom. Funct. Anal., 1993, 3, 107-156.

\bibitem{20}
X. Carvajal, M. Panthee, Sharp well-posedness for a coupled system of mKdV-type equations, 2020, https://arxiv.org/abs/2003.12619.

	
\bibitem{T1}
C. Foias, R. Temam, Gevrey class regularity for the solutions of the Navier–Stokes equations, J. Funct. Anal., 1989, 87, 359-369. 

\bibitem{13}
Z. Grujic, H. Kalisch, Local well-posedness of the generalized Korteweg–de Vries equation in spaces of analytic functions, Diff. Inte. Equ., 2002, 15, 1325-1334.

\bibitem{14}
N. Hayashi, Solutions of the (generalized) Korteweg–de Vries equation in the Bergman and Szego spaces on a sector, Duke Math. J., 1991, 62, 575-591.

\bibitem{A1}
C.E. Kenig, G. Ponce, L. Vega, On the Cauchy problem for the Korteweg–deVries equation in Sobolev spaces of negative indices, Duke Math. J., 1993, 71, 1-20.

\bibitem{5}
M Panthee, M Scialom, On the Cauchy problem for a coupled system of KdV equations: critical case. Adv. Diff. Equ., 2008, 13(1-2), 1-26.

\bibitem{21}
M. Shan, L. Zhang, Lower bounds on the radius of spatial analyticity for the 2D generalized Zakharov-Kuznetsov equation, J. Math. Anal. Appl., 2021, 501, 125218. 

\bibitem{18}
P. Souganidis, W. Strauss, Instability of a class of dispersive solitary waves, Proc. Roy. Soc. of Edinburgh, 114A, 1990, 195-212. 

\bibitem{Boukarou5}	Kh. Zennir, A. Boukarou, R.N. Alkhudhayr {\em Global Well-Posedness for Coupled System of mKdV Equations in Analytic Spaces}
Journal of Function Spaces, 2021, Article ID 6614375.

\end{thebibliography}
\end{document}